\documentclass[final]{siamart171218}
\usepackage{amssymb,amsfonts}
\usepackage{bm,upgreek,nicefrac,mathtools,pdf14}
\usepackage[mathscr]{eucal}
\usepackage{dsfont}

\usepackage{enumitem}
\makeatletter
\def\namedlabel#1#2{\begingroup
    #2%
    \def\@currentlabel{#2}%
    \phantomsection\label{#1}\endgroup}
\makeatother

\newlist{describe}{description}{1}
\setlist[describe,1]{%
  font=\normalfont\textbf,
  itemindent=0pt,
  wide,
  itemsep=0pt,topsep=2pt}
\newsiamremark{remark}{Remark}
\newsiamremark{assumption}{Assumption}
\newsiamremark{notation}{Notation}
\newsiamremark{example}{Example}
\newcommand{\vertiii}[1]{{\left\vert\kern-0.25ex\left\vert\kern-0.25ex\left\vert #1 
    \right\vert\kern-0.25ex\right\vert\kern-0.25ex\right\vert}}
\newcommand{\rc}{{\mathscr{R}}}    
\newcommand{\cIn}{\Ind_{\{\abs{z}\le 1\}}}

\newcommand{\Uadm}{\bm{\mathfrak{Z}}}
\newcommand{\Act}{\mathcal{Z}}
\newcommand{\Usm}{\mathfrak{V}_{\mathsf{sm}}}
\newcommand{\Ussm}{\mathfrak{V}_{\mathsf{ssm}}}
\newcommand{\Inv}{\mathcal{M}} 

\newcommand{\Ag}{{\mathcal{A}}}  
\newcommand{\Lg}{{\mathcal{L}}}    
\newcommand{\Bor}{{\mathfrak{B}}}  
\newcommand{\cB}{{\mathcal{B}}}  
\newcommand{\sB}{{\mathscr{B}}}  
\newcommand{\Cc}{{\mathcal{C}}}   
\newcommand{\sC}{{\mathscr{C}}}   
\newcommand{\sF}{{\mathfrak{F}}}   
\newcommand{\eom}{{\mathcal{G}}} 
\newcommand{\cI}{{\mathcal{I}}}  

\newcommand{\Lp}{{L}}            
\newcommand{\Lpl}{L_{\text{loc}}}            
\newcommand{\Pm}{{\mathcal{P}}}  
\newcommand{\cN}{{\mathcal{N}}}  

\newcommand{\Lyap}{{\mathscr{V}}}  

\newcommand{\RR}{\mathds{R}}
\newcommand{\NN}{\mathds{N}}

\newcommand{\Rd}{{\mathds{R}^{d}}}
\DeclareMathOperator{\Exp}{\mathbb{E}}
\DeclareMathOperator{\Prob}{\mathbb{P}}
\newcommand{\D}{\mathrm{d}}
\newcommand{\E}{\mathrm{e}}
\newcommand{\Ind}{\mathds{1}}   

\newcommand{\Sob}{{\mathscr W}}    
\newcommand{\Sobl}{{\mathscr W}_{\text{loc}}} 

\newcommand{\df}{\coloneqq}

\DeclareMathOperator*{\trace}{trace}

\DeclareMathOperator*{\Argmin}{Arg\,min}

\newcommand{\grad}{\nabla}
\newcommand{\uuptau}{{\Breve\uptau}}

\newcommand{\abs}[1]{\lvert#1\rvert}
\newcommand{\norm}[1]{\lVert#1\rVert}
\newcommand{\babs}[1]{\bigl\lvert#1\bigr\rvert}

\newcommand{\bnorm}[1]{\bigl\lVert#1\bigr\rVert}

\usepackage{color}
\definecolor{dmagenta}{rgb}{.4,.1,.5}
\definecolor{dblue}{rgb}{.0,.0,.5}
\definecolor{mblue}{rgb}{.0,.0,.8}
\definecolor{ddblue}{rgb}{.0,.0,.4}
\definecolor{dred}{rgb}{.6,.0,.0}
\definecolor{dgreen}{rgb}{.0,.5,.0}
\definecolor{Eeom}{rgb}{.0,.0,.5}

\newcommand{\TheTitle}{Ergodic control of a class of jump diffusions} 
\newcommand{\TheAuthors}{A. Arapostathis, L. Caffarelli, G. Pang, and Y. Zheng}
\headers{\TheTitle}{\TheAuthors}

\title
{{Ergodic control of a class of jump diffusions\\[2pt]
with  finite L\'evy measures and rough kernels}}

\author{Ari Arapostathis\thanks{Department of Electrical and Computer
Engineering, The University of Texas at Austin, 2501 Speedway, EER 7.824,
Austin, TX 78712 (\email{ari@ece.utexas.edu}).}
\and Luis Caffarelli\thanks{Department of Mathematics,
The University of Texas at Austin, 2515 Speedway, RLM 10.150,
Austin, TX 78712 (\email{caffarel@math.utexas.edu}).}
\and Guodong Pang\thanks{The Harold and Inge Marcus Department of Industrial and
Manufacturing Engineering,
College of Engineering,
Pennsylvania State University,
University Park, PA 16802 (\email{gup3@psu.edu}, \email{yxz282@psu.edu}).}
\and Yi Zheng\footnotemark[3]}

\begin{document}
\maketitle

\begin{abstract}
We study the ergodic control problem for a class of jump diffusions in $\Rd$,
which are controlled through the drift with bounded controls.
The L$\acute{\text{e}}$vy measure is finite, but has
no particular structure---it can be anisotropic and singular.
Moreover, there is no blanket ergodicity assumption for the controlled process.
Unstable behavior is  `discouraged' by
the running cost which satisfies a mild coercive hypothesis
(i.e., is near-monotone).
We first study the problem in its weak formulation as an optimization problem
on the space of infinitesimal ergodic occupation
measures, and derive the Hamilton--Jacobi--Bellman equation
under minimal assumptions on the parameters, including verification of optimality
results, using only analytical arguments.
We also examine the regularity of invariant measures.
Then, we address the jump diffusion model, and obtain a
complete characterization of optimality.  
\end{abstract}



\section{Introduction}\label{S1}
Optimal control of jump diffusions has recently attracted much
attention from the control community, primarily due to its applicability
to queueing networks, mathematical finance
\cite{cont-tankov}, image processing \cite{gilboa-osher}, etc.
Many results for the discounted problem are available in \cite{BenLi-84},
including the game theoretic setting, and different applications are discussed.
However, studies of the ergodic control problem are rather scarce.
Ergodic control of reflected jump diffusions over a bounded domain
can be found in \cite{Menaldi-97}.
The ergodic control problem in $\Rd$ is studied in \cite{Menaldi-99}, albeit
under very strong blanket stability assumptions. 
We should also mention here the treatment of the impulse control problem
in \cite{Bayraktar-13,Davis-10,Liu-18}. 

Our work in this paper is motivated from ergodic control problems
for multiclass stochastic networks in the Halfin--Whitt regime,
under service interruptions.
For this model, the pure jump process driving the limiting queueing process
is compound Poisson (see Theorem~3.2 in \cite{APS17}), with
a L\'evy measure that is anisotropic, and in general, singular with respect to
the Lebesgue measure.
In fact, the jumps are biased towards a given direction, and thus
the L\'evy measure has no symmetry whatsoever.
We assume that the running cost is coercive, also known as
near-monotone (see \cref{E-nm}), and do not impose any blanket stability
hypotheses on the controlled jump diffusion.
We treat a general class of jump diffusions which is abstracted from
diffusion approximations
of stochastic networks, and
whose controlled infinitesimal generator has the form
\begin{align}\label{E-Ag}
\Ag u(x,z) &\;\df\; \sum_{i,j} a^{ij}(x)
\frac{\partial^2u}{\partial x_i \partial x_j}(x) + \sum_{i} b^i(x,z)
\frac{\partial u}{\partial x_i}(x) \\
&\mspace{100mu} + \int_{\RR^d}
\bigl( u(x+y) - u(x) - \Ind_{\{\abs{y}\le1\}} \langle y, \nabla u(x)\rangle\bigr)
\,\nu_x(\D y)\,. \nonumber
\end{align}
Here, $z$ is a control parameter that lives in a compact metric space $\Act$,
and $\nu_x(\D{y})$ is a finite Borel measure on $\Rd$ for each $x$,
while $x\mapsto\nu_x(A)$ is a Borel measurable function for each Borel set $A$.
Throughout the paper, we assume that $d\ge2$.
The coefficients of $\Ag$ are assumed to satisfy the following.

\begin{assumption}\label{A1.1}
\begin{itemize}
\item[(a)]
The matrix $a=[a^{ij}]$ is symmetric, positive definite,
and locally Lipschitz continuous.
The drift $b\colon\Rd\times\Act\to\Rd$ is continuous.
\item[(b)]
The map $x\mapsto \bm\nu(x)\df \nu_x(\Rd)$ is locally bounded.
\item[(c)]
the map $x\mapsto\nu_x(K-x)$ is bounded on $\Rd$
for any fixed compact set $K\subset\Rd$.
\end{itemize}
\end{assumption}

The generator 
$\Ag$ in \cref{E-Ag} covers a variety of models of jump diffusions
which appear in the literature
\cite{Bass-04,Bogachev-14,Foondun-09,GS72,Skorokhod-89}.
Note also that the `jump rate' $\bm\nu(x)$ is allowed
to be state dependent as in \cite{Locherbach-17}.
The hypotheses in \cref{A1.1} are quite general, and do not
imply the existence of a controlled process with generator $\Ag$.
Our main goal in this paper is to establish  general results for
ergodic control of jump diffusions governed for this class of operators.
To accomplish this, we first state the ergodic control problem for
the operator $\Ag$ as a convex optimization problem over the set
of infinitesimal ergodic occupation measures.
We then proceed to study the ergodic Hamilton--Jacobi--Bellman (HJB) equation
via analytical methods, without assuming that the martingale problem
for $\Ag$ is well posed.
This of course precludes arguments that utilize stochastic representations
of solutions of elliptic equations.
Later, in \cref{S4}, we specialize these results to a fairly general model
of controlled jump diffusions with finite L\'evy measure.

It is well known that the standard method of
deriving the ergodic HJB
on $\Rd$ is based on the vanishing discount approach,
and relies crucially on structural properties that permit uniform estimates
for the gradient (e.g., viscous equations in $\Rd$), or the
Harnack property.
Recent work on nonlocal equations has resulted in important regularity
results \cite{bass-2009,Bjorland-12,caffarelli-silvestre-regu,caffarelli-silvestre-approx} that should prove very valuable in studying control problems.
However, most of this work concerns L\'evy jump processes whose kernel has
a `nice' density resembling that of a fractional Laplacian.
For the problem at hand, even though the L\'evy measure $\nu_x$ is finite,
and there is a non-degenerate Wiener process component, 
the L\'evy measure is anisotropic, and could be singular
\cite[Section~3.2]{APS17}.
As a result, there is no hope for the Harnack property for positive solutions to hold
as the following example shows.

\begin{example}
Consider an operator $\Ag$ in $\RR^2$, with $a$ the identity matrix,
$b=(3, 0)$,
and $\nu=\nu_x$ a Dirac mass at $\Tilde{x}=(3,0)$.
Let $f_\epsilon\in\Cc^2(\RR^2)$, with $\epsilon\in(0,1)$,
be defined in polar coordinates by
\begin{equation*}
f_\epsilon(r,\theta)\;\df\; -\log (r)\,\Ind_{\{r\ge\epsilon\}}
+\Bigl(\tfrac{3}{4} - \tfrac{r^2}{\epsilon^2} + \tfrac{r^4}{4\epsilon^4}
- \log(\epsilon)\Bigr)\,\Ind_{\{r<\epsilon\}}\,.
\end{equation*}
This function is used in \cite[p.~111]{ProtWein-84} to exhibit a family
of positive superharmonic functions for the Laplacian that violates the
Harnack property.
Let $u_\epsilon$ be a function which agrees with
$f_\epsilon$ on the unit ball $B_1$ centered at $0$, and takes the values
$u_\epsilon(\Tilde{r},\Tilde{\theta})\;=\;
\bigl( \tfrac{4}{\epsilon^2} - \tfrac{4\Tilde{r}^2}{\epsilon^4}
+ f_\epsilon(\Tilde{r},\Tilde{\theta})\bigr)\,\Ind_{\{\Tilde{r}<\epsilon\}}$
on the unit ball $B_1(\Tilde{x})$
centered at $\Tilde{x}$, when expressed
in polar coordinates $(\Tilde{r},\Tilde{\theta})$ which are centered at $\Tilde{x}$.
Let $u_\epsilon$ take any nonnegative value elsewhere in $\RR^2$.
Then $u_\epsilon$ is nonnegative on $\RR^2$ and satisfies
$\Ag u_\epsilon=0$ in $B$.
However, $\frac{u_\epsilon(0,\theta)}{u_\epsilon(\E^{-1},\theta)}=-\log(\epsilon)$,
and thus the family violates the Harnack property for $\Ag$.
\end{example}


Under the general hypotheses of \cref{A1.1}, even if
the operator $\Ag$ is the generator of a  Markov process,
the process might not be regular,
or, in case it is positive recurrent,
the mean hitting times to an open ball might not be locally bounded.
In the latter case, it is futile to search for solutions to the
ergodic HJB equation, even in a viscosity sense.
In \cref{S3}, we add two hypotheses to address these pathologies.
The first (see \ref{H1}), is the Feller--Has$^{_{^{^\prime}}}\!$minski\u{\i} criterion
for a diffusion process with generator
$\Ag$ to be \emph{regular} (or \emph{conservative}, or \emph{non-explosive}),
which requires
that the equation $\Ag u - u =0$ has no bounded positive solutions on $\Rd$.
This property is equivalent to regularity, and it is clear from the proof
of this equivalence
in \cite[Theorem~4.1]{Hasm-60} that the equation can be replaced
by $\Ag u - \alpha u =0$ for $\alpha>0$.
The second hypothesis, \ref{H2}, states that under some stationary Markov control there
exists a nonnegative solution $\Lyap$ to the Lyapunov equation
$\Ag\Lyap \le C\Ind_{\sB} - \rc$, where $\rc$ is the running cost,
$\sB$ is a ball, and
$C$ is a constant.
Hypothesis \ref{H2} can be relaxed under certain assumpions on $\nu_x$ (see
\cref{T3.8}).

The paper is organized as follows.
In \cref{S1.1} we summarize the notation we use.
\Cref{S2} states the ergodic control problem, in a weak sense,
as a convex optimization problem over the set of infinitesimal
ergodic occupation measures for the operator $\Ag$,
and shows that optimality is attained.
Regularity properties of infinitesimal invariant measures are in \cref{S2.3}.
\Cref{S3} is devoted to the study of the HJB equation under 
\ref{H1}--\ref{H2} mentioned above.
In \Cref{S4} we study a class of jump diffusions, which is abstracted
from the limiting diffusions encountered in
stochastic networks under service interruptions.

\subsection{Notation}\label{S1.1}
The standard Euclidean norm in $\RR^{d}$ is denoted by $\abs{\,\cdot\,}$,
and $\langle\,\cdot\,,\cdot\,\rangle$ denotes the inner product. 
Given two real numbers $a$ and $b$, the minimum (maximum) is denoted by $a\wedge b$ 
($a\vee b$), respectively.
The closure, boundary, complement, and the indicator function
of a set $A\subset\Rd$ are denoted
by $\Bar{A}$, $\partial{A}$, $A^{c}$, and $\Ind_A$,  respectively.
We denote by $\uptau(A)$ the \emph{first exit time} of the process
$X$ from a set $A\subset\RR^{d}$, defined by
$\uptau(A) \;\df\;\inf\;\{t>0\,\colon\,X_{t}\not\in A\}$.
The open ball of radius $R$ in $\RR^{d}$, centered at the origin,
is denoted by $B_{R}$, and we let $\uptau_R\;\df\;\uptau(B_{R})$,
and $\uuptau_R\df \uptau(B^{c}_{R})$.
The Borel $\sigma$-field of a topological space $E$ is denoted by $\Bor(E)$,
and $\Pm(E)$ denotes the set of probability measures on $\Bor(E)$.

For a domain $Q\subset\RR^{d}$,
the space $\Cc^{k}(Q)$ ($\Cc^{\infty}(Q)$), $k\ge 0$,
refers to the class of all real-valued functions on $Q$ whose partial
derivatives up to order $k$ (of any order) exist and are continuous,
while $\Cc_{\mathrm{c}}^k(Q)$ ($\Cc^{k}_{b}(Q)$)
denote the subsets of $\Cc^{k}(Q)$, 
consisting of functions that have compact support
(whose partial derivatives are bounded
in $Q$).
The space $\Lp^{p}(Q)$, $p\in[1,\infty)$, stands for the Banach space
of (equivalence classes of) measurable functions $f$ satisfying
$\int_{Q} \abs{f(x)}^{p}\,\D{x}<\infty$,
and $\Lp^{\infty}(Q)$ is the
Banach space of functions that are essentially bounded in $Q$.
We denote the usual norm on this space by  $\norm{f}_{\Lp^p(Q)}$,
$p\in[1,\infty]$.
The standard Sobolev space of functions on $Q$ whose generalized
derivatives up to order $k$ are in $\Lp^{p}(Q)$, equipped with its natural
norm, is denoted by $\Sob^{k,p}(Q)$, $k\ge0$, $p\ge1$.
In general, if $\mathcal{X}$ is a space of real-valued functions on $Q$,
$\mathcal{X}_{\mathrm{loc}}$ consists of all functions $f$ such that
$f\varphi\in\mathcal{X}$ for every $\varphi\in\Cc_{c}^{\infty}(Q)$.
In this manner we obtain, for example, the space $\Sobl^{2,p}(Q)$.

We adopt the notation
$\partial_{i}\df\tfrac{\partial~}{\partial{x}_{i}}$ and
$\partial_{ij}\df\tfrac{\partial^{2}~}{\partial{x}_{i}\partial{x}_{j}}$
for $i,j\in\{1,\dotsc,d\}$, and we
often use the standard summation rule that
repeated subscripts and superscripts are summed from $1$ through $d$.

\section{The convex analytic formulation}
\label{S2}

Define
$\Lg\colon \Cc^2(\RR^d) \to \Cc(\RR^d \times \Act)$ by
\begin{equation*}
\Lg u(x,z) \;\df\;  a^{ij}(x)
\partial_{ij}u(x) + \Hat{b}^i(x,z)\partial_i u(x)\,,
\end{equation*}
with $\Hat{b}(x,z) \df b(x,z)+ \int_\Rd z\,\cIn\nu_x(\D z)$,
and let
\begin{equation*}
\cI u(x) \;\df\;  \int_\Rd  \bigl(u(x+y)-u(x)\bigr)\,\nu_x(\D{y})\,,
\end{equation*}
provided that the integral is finite. 
Thus $\Ag u(x,z) = \Lg u(x,z) +\cI u(x)$.
With $z\in\Act$ treated as a parameter, we define
$\Lg_z u(x)\df \Lg u(x,z)$, and $\Ag_z u(x)\df \Ag u(x,z)$.

Let $\cB(\Rd,\Act)$ denote the set of Borel measurable maps
$v\colon\Rd\to\Act$.
Such a map $v$ is called a \emph{stationary Markov control}, and we use
the symbol $\Usm$ to denote this class of controls. 
For $v\in\Usm$, we use the simplified notation $b_v(x)\df b\bigl(x,v(x)\bigr)$,
and define $\Ag_v$, $\rc_v$ and $\varrho_v$ analogously.

We augment the class $\Usm$ by  
adopting the well-known \emph{relaxed control} framework \cite[Section~2.3]{ABG12}.
According to this relaxation, controls take values in
$\Pm(\Act)$, the latter denoting the set of probability measures on $\Act$ under
the Prokhorov topology.
Thus, a control $v\in\Usm$ may be viewed as a kernel on $\Pm(\Act)\times\Rd$,
which we write as $v(\D{z}\!\mid\! x)$. 
We extend the definition of $b$ and $\rc$, without changing the notation, i.e.,
we let $b_v(x)\df\int_{\Act} b(x,z)\,v(\D{z}\!\mid\! x)$,
and analogously for $\rc_v$.
We endow $\Usm$ with the
topology that renders it a compact metric space, referred to
as the \emph{topology of Markov controls} \cite[Section~2.4]{ABG12}.
A control is said to be \emph{precise} if it is a measurable map
from $\Rd$ to $\Act$, i.e., if it agrees with the definition in the
preceding paragraph.
It is easy to see that this relaxation preserves \cref{A1.1}.

\subsection{The ergodic control problem for the operator
\texorpdfstring{$\Ag\,$}{}}\label{S2.1}

We fix a countable dense subset $\sC$ of $\Cc_0^2(\RR^d)$ consisting
of functions with compact supports.
Here, $\Cc_0^2(\RR^d)$ denotes the Banach space of functions $f\colon\Rd\to\RR$
that are twice continuously differentiable and their derivatives up to
second order vanish at infinity.

\begin{definition}\label{D2.1}
A probability measure $\mu_{v}\in\Pm(\Rd)$, $v\in\Usm$, is called
\emph{infinitesimally invariant} under $\Ag_v$ if
\begin{equation}\label{E-infmeas}
\int_{\RR^d}^{}\Ag_v f(x)\,\mu_v(\D{x}) \;=\; 0 \quad\forall\,f\in\sC\,.
\end{equation}
If such a $\mu_v$ exists, then we say that
$v$ is a \emph{stable} control, and define the
\emph{(infinitesimal) ergodic occupation measure} $\uppi_v\in \Pm(\RR^d\times \Act)$ by 
$\uppi_v(\D{x},\D{z}) \df \mu_v(\D{x})\,v(\D{z}\,|\, x)$.
We denote by $\Ussm$, $\Inv$, and $\eom$, the sets of
stable controls, infinitesimal invariant probability measures, and
ergodic occupation measures, respectively.
\end{definition}

\begin{remark}
In \cref{D2.1} we select $\sC$ as the function space,
deviating from common practice, where this is selected as $\Cc_0^\infty(\Rd)$,
the space of smooth functions vanishing at infinity.
In general, there is no uniqueness of solutions to \cref{E-infmeas}
\cite{Shaposhnikov-08}.
For the relation between infinitesimally invariant measures
and invariant probability measures for diffusions we refer the reader to
\cite{Bogachev-02}.
Note also, that as shown in \cite{PE82}, in order to assert
that $\mu_v$ is an invariant probability measure for a Markov process
with generator $\Ag_v$,  it suffices to verify \cref{E-infmeas} for
a dense subclass of the domain of $\Ag_v$ consisting of functions
such that the martingale problem is well posed.
\end{remark}

It follows from \cref{D2.1} that $\uppi\in\Pm(\RR^d\times \Act)$
is an ergodic occupation
measure if and only if
$\int_{\RR^d\times\Act}^{} \Ag_zf(x)\,\uppi(\D{x},\D{z}) = 0$
for all $f\in\sC$.
It is also easy to show that the set of ergodic occupation measures
$\eom$ is a closed and convex subset of $\Pm(\RR^d\times \Act)$
(see \cite[Lemma~3.2.3]{ABG12}).

Let $\rc\colon\RR^d\times\Act \mapsto\RR_+$ be a continuous
function, which we refer to as the \emph{running cost} function.
The ergodic control problem for $\Ag$ seeks to minimize
$\uppi(\rc)=\int\rc\,\D\uppi$ over $\uppi\in\eom$.
Thus, the optimization problem is an infinite dimensional linear
program.
We define $\varrho_* \df \inf_{\uppi\in\eom}\,\uppi(\rc)$, and assume,
of course, that this is finite.
Also for $v\in\Ussm$, we let $\varrho_v\df\uppi_v(\rc)$, and we say
that $v$ is \emph{optimal} if $\varrho_v=\varrho_*$.
We seek to obtain a full characterization of optimal controls
via the study of the dual problem, and this leads to the HJB equation.
For more details on this linear programming formulation
see Section~4 in \cite{BhBo}.

\subsection{Well posedness of the control problem}\label{S2.2}
We impose a structural assumption on the running cost which renders
the optimization problem well posed.
We say that a function
$h: \RR^d\times \Act \to \RR_+$ is \emph{coercive relative to a constant $c\in\RR$},
if there exists a constant $\epsilon>0$,
such that the set
$\{x\in \Rd\,\colon\inf_{z\in\Act}\; h(x,z)\le c+\epsilon\}$ is bounded (or empty).

Throughout the paper, we assume that the running cost is coercive relative to
$\varrho_*$,
and we fix a ball $\sB_\circ$ and a constant $\epsilon_\circ$ such that
$\rc(x,z) > \varrho_*+2 \epsilon_\circ$ on $\sB_\circ^c$.
Naturally, this property depends on $\varrho_*$,
but note that, since $\varrho_*<\infty$, it is always satisfied if
the running cost is inf-compact on $\Rd\times\Act$.
Coerciveness of $\rc$ relative to $\varrho_*$
is also known as \emph{near-monotonicity} in the literature,
and it is often written as
\begin{equation}\label{E-nm}
\liminf_{\abs{y}\rightarrow\infty}\; \inf_{z\in\Act}\; \rc(y,z) \;>\;\varrho_*\,.
\end{equation}

We state the following theorem, which follows easily by mimicking the proofs
of Lemma~3.2.11 and Theorem~3.4.5 in \cite{ABG12}.

\begin{theorem}
The map $\uppi\mapsto \uppi(\rc)$ attains its minimum in $\eom$.
\end{theorem}

\subsection{Regularity properties of infinitesimal invariant measures}\label{S2.3}
In this section we establish regularity properties
of the densities of infinitesimal invariant
probability measures.
Recall the notation $\bm\nu(x)=\nu_x(\Rd)$ introduced in \cref{A1.1}.
We need the following definition.

\begin{definition}\label{D2.4}
We decompose $\Ag_z = \widetilde\Lg_z + \widetilde\cI$, with
\begin{equation*}
\widetilde\Lg_z u(x) \;\df\; \Lg_z u(x) - \bm\nu(x)u(x)\,,\qquad\text{and}\quad
\widetilde\cI u(x)\;\df\; \int_\Rd u(x+y)\,\nu_x(\D{y})\,.
\end{equation*}
\end{definition}

\begin{theorem}\label{T2.5}
Every $\mu\in\Inv$ has a density $\upphi=\upphi[\mu]$ which
belongs to $\Lpl^p(\Rd)$ for any $p\in\bigl[1,\frac{d}{d-2}\bigr)$,
and is strictly positive.
In addition, if $\nu_x$ is translation invariant and has compact support,
then, for any $\beta\in(0,1)$,
there exists a constant $\Bar{C} = \Bar{C}(\beta,R)$, such that
\begin{equation}\label{ET2.5A}
\abs{\upphi(x) - \upphi(y)} \le \Bar{C}\,\abs{x-y}^{\beta}
\qquad\forall\,x,y\in B_R\,.
\end{equation}
\end{theorem}

\begin{proof}
As shown in \cite[Theorem~2.1]{Bogachev-01}, if in some domain
$Q\subset\Rd$, a probability measure
$\mu$ satisfies
\begin{equation}\label{PT2.5A}
\int_{Q} a^{ij} \partial_{ij} f\,\D\mu \;\le\; C\,\sup_{Q}\bigl(\abs{f} + \abs{\grad f}
\bigr)\qquad\forall f\,\in\Cc_c^{\infty}(Q)
\end{equation}
for some constant $C$, then $\mu$ has a density which belongs to
$\Lpl^p(Q)$ for every $p\in[1,d')$, where $d'=\frac{d}{d-1}$.
It is straightforward to verify, using \cref{A1.1}, that a bound of the form
\cref{PT2.5A} holds for any
$\mu\in\Inv$ on any bounded domain $Q$.
It follows that the density $\upphi$ of $\mu$ is in
$\Lpl^p(\Rd)$ for any $p\in[1,d')$, and that it
is a generalized solution to the equation
\begin{align}\label{PT2.5B}
\sum_{i,j} \int_\Rd
\bigl(a^{ij}(x)&\partial_{j}\upphi(x) +
\bigl(\partial_{j}a^{ij}(x) - \Hat{b}^{i}_{v}(x)\bigr)\upphi(x)\bigr)\partial_{i} f(x)\,
\,\D{x}\\
&-\int_\Rd\bm\nu(x)\upphi(x)f(x)\,\D{x}
\;=\; -\int_\Rd\int_\Rd f(x+y) \nu_x(\D{y})\,\upphi(x)\,\D{x}\,,\nonumber
\end{align}
for $f\in\Cc_c^\infty(\Rd)$.
By \cref{PT2.5B}, $\upphi$ is a supersolution to
\begin{equation}\label{PT2.5C}
\widetilde\Lg^{*}_v\upphi(x) \;\df\; \partial_{i}
\bigl(a^{ij}(x)\partial_{j}\upphi(x) +
\bigl(\partial_{j}a^{ij}(x) - \Hat{b}^{i}_v(x)\bigr)\upphi(x)\bigr)
-\bm\nu(x)\upphi(x)\;=\; 0\,.
\end{equation}
Therefore, by
the estimate for supersolutions in \cite[Theorem~8.18]{GilTru},
we deduce that $\upphi\in\Lpl^p(\Rd)$
for any $p\in\bigl[1,\frac{d}{d-2}\bigr)$, and that it is strictly positive.
Note that this theorem assumes
that the supersolution is in $\Sobl^{1,2}(\Rd)$, but this is unnecessary.
The theorem is valid for functions in $\Sobl^{1,p}(\Rd)$ for any $p>1$,
as seen from the results in Section~5.5 of \cite{Morrey-66}, or one can use
the mollifying technique in \cite[Theorem~5.3.4]{ABG12} to show this.

Now suppose that $\nu_x$ is translation invariant and has compact support.
Let $\widehat{\cI}\upphi(x)\df\int_\Rd \upphi(x-y)\,\nu(\D{y})$.
Then \cref{PT2.5B} takes the form
$\widetilde\Lg^{*}_v\upphi(x) = -\widehat{\cI}\upphi(x)$.
The operator $\widetilde\Lg^{*}_v$ satisfies the
hypotheses of Theorem~5.5.5$^\prime$ in \cite{Morrey-66}, which asserts that
$\upphi$ satisfies
\begin{equation}\label{PT2.5F}
\norm{\upphi}_{\Sob^{1,q}(B_R)}
\;\le\; \kappa(p,R)\,\bigl( \norm{\widehat{\cI}\upphi}_{\Lp^{p}(B_{2R})}
+ \norm{\upphi}_{\Lp^{1}(B_{2R})}\bigr)\qquad\forall\,p>1\,,
\end{equation}
with $q=q(p)\df \frac{d p}{d-p}$, and a constant $\kappa(p,R)$ that
depends also on $d$, $\bm\nu$, and the bounds in \cref{A1.1}.
Without loss of generality, suppose that $\nu$ is supported on
a ball $B_{R_\circ}$.
By Minkowski's integral inequality we have
\begin{equation}\label{PT2.5G}
\norm{\widehat{\cI}\upphi}_{\Lp^{p}(B_{2R})}\;\le\;
\bm\nu\,\norm{\upphi}_{\Lp^{p}(B_{2R+R_\circ})}\,.
\end{equation}
On the other hand, by the Sobolev embedding theorem,
$\Sob^{1,q}(B_R)\hookrightarrow \Lp^r(B_R)$ is a continuous embedding
for $q\le r \le\frac{qd}{d-q}$ and $q<d$, and
$\Sob^{1,q}(B_R)\hookrightarrow \Cc^{0,r}(\overline{B}_R)$ is compact
for $r<1-\frac{d}{q}$ and $q>d$.
Therefore, starting say from $p=\frac{d}{d-1}$, we deduce
by repeated applications of \cref{PT2.5F}--\cref{PT2.5G}, and Sobolev embedding,
that
$\upphi\in\Sobl^{1,q}(\Rd)$ for any $q>1$, which implies \cref{ET2.5A}.
\end{proof}

\begin{remark}
The assumption that $\nu_x$ is translation invariant in \cref{T2.5} is sharp.
Consider a jump diffusion with $\upsigma=\sqrt{2}$, $b(x)= x$, $g(x,\xi) = - x$,
and $\bm\nu=1$.
Then $\Ag = \Delta - 1 + \delta_0$, where $\delta_0$ denotes the Dirac mass at $0$.
It can be easily verified that the diffusion is geometrically ergodic by
employing the Lyapunov function $\Lyap(x)=\abs{x}^2$.
The density of the invariant measure $\upphi$ satisfies
$\int \sum_{ij}(\partial_i\upphi)(\partial_j f) + \int\upphi f = f(0)$
for all $f\in\Cc_c^\infty(\Rd)$,
and thus it is a solution of $-\Delta \upphi+ \upphi = \delta_0$
(viewed in the sense of distributions $\mathcal{D}'(\Rd)$).
However, as shown in \cite{Sato-95}, every positive solution $\upphi$ of
this equation, which vanishes at infinity,
satisfies $\upphi(x)\sim \Gamma(x)$ as $x\to0$,
where $\Gamma$ denotes the fundamental solution of $-\Delta$ in $\Rd$.
Thus the density of the invariant measure in the vicinity of $x=0$
is not any better than what is claimed in the first step in the proof,
which shows that it belongs to $\Lpl^p(\Rd)$ for $p<\frac{d}{d-2}$.
One can select the jumps to induce multiple such singularities, and
generate very pathological examples.
Thus, in general, the hypothesis that $\nu_x$ is translation invariant
cannot be relaxed, unless we assume that $\nu_x$ has a suitable density
as shown in \cref{C2.8} below.
\end{remark}

\begin{definition}\label{D2.7}
We say that $\nu_x$ has \emph{locally compact support} if there exists
an increasing map
$\gamma\colon(0,\infty)\to(0,\infty)$ such that
$\nu_x(x+B_{\gamma(R)}^c)=0$ for all $x\in B_R$.
Let $\widehat{\gamma}(R)\df R+\gamma(R)$.
It follows from this definition
 that $B_{\widehat\gamma(R)}$ contains the support of $\nu_x$
for all $x\in B_R$.
\end{definition}

\begin{corollary}\label{C2.8}
Assume that $\nu_x$ has locally compact support, and that
it has a density $\psi_x\in\Lpl^{p_1}(\Rd)$ for some $p_1>\frac{d}{2}$,
satisfying the following:
for some $p_2\in\bigl(1,\frac{d}{d-2}\bigr)$, it holds that
\begin{equation*}
\int_{B_{\gamma(R)}}\biggl(\int_{B_{\widehat\gamma(R)}}
\abs{\psi_x(y)}^{p_i}\,\D{y}\biggr)^{\frac{1}{p_i-1}}
\,\D{x} \;<\;\infty\,,\quad i=1,2\,,\quad\forall R>0\,.
\end{equation*}
Then \cref{ET2.5A} holds.
\end{corollary}

\begin{proof}
Note that
\begin{align*}
\int_{\Rd}\biggl(\int_{\Rd} f(x+y) \psi_x(y)\,\D{y}\biggr) \upphi(x)\,\D{x}
&\;=\;
\int_{\Rd}f(z)\biggl(\int_{\Rd} \psi_{z-y}(y)\,\upphi(z-y)\,\D{y}\biggr) \D{z}\\
&\;=\;
\int_{\Rd}f(z)\biggl(\int_{\Rd} \psi_{z}(z-a)\,\upphi(a)\,\D{a}\biggr) \D{z}\,.
\end{align*}
Therefore,
$\widehat{\cI}\upphi(x)=\int_{\Rd} \psi_{a}(x-a)\,\upphi(a)\,\D{a}$.
By the Minkowski integral inequality and the H\"older inequality,
we obtain
\begin{align*}
\bnorm{\widehat{\cI}h(z)}_{\Lp^p(B_R)} &\;=\;
\biggl(\int_{B_R}\biggl(\int_{B_{\gamma(R)}}
\psi_{a}(z-a)\,\abs{h(a)}\,\D{a}\biggr)^p\,\D{z}\biggr)^{\nicefrac{1}{p}}\\
&\;\le\;
\int_{B_{\gamma(R)}}\abs{h(a)}
\biggl(\int_{B_R} \abs{\psi_{a}(z-a)}^p\,\D{z}\biggr)^{\nicefrac{1}{p}}
\,\D{a}\\
&\;\le\;
\norm{h}_{\Lp^p(B_{\gamma(R)})}
\biggl(\int_{B_{\gamma(R)}}
\biggl(\int_{B_R} \abs{\psi_{a}(z-a)}^p\,\D{z}\biggr)^{\nicefrac{1}{(p-1)}}\,\D{a}
\biggr)^{\nicefrac{(p-1)}{p}}\\
&\;\le\;
\norm{h}_{\Lp^p(B_{\gamma(R)})}
\biggl(\int_{B_{\gamma(R)}}
\bnorm{\psi_{a}}_{\Lp^p(B_{\widehat\gamma(R)})}^{\nicefrac{p}{(p-1)}}\,\D{a}
\biggr)^{\nicefrac{(p-1)}{p}}\,.
\end{align*}

Therefore, the map $\widehat{\cI}h$
is a linear mapping from $\Lp^{p_1}(B_{\gamma(R)})\cup\Lp^{p_2}(B_{\gamma(R)})$ into
$\Lp^{p_1}(B_R)\cup\Lp^{p_2}(B_R)$
and satisfies  
\begin{equation*}
\babs{ \bigl\{x\in B_R \,\colon \abs{\widehat{\cI}h(x)}>t\bigr\}} \;\le\; C \,
\frac{\norm{h}_{\Lp^{p_i}(B_{\gamma(R)})}}{t^{\,p_i}}
\end{equation*}
for some constant $C$,
for all $h\in\Lp^{p_i}(B_R)$, $i=1,2$.
Here, $\abs{A}$ denotes the Lebesgue measure of a set $A$.
Thus, by the Marcinkiewicz interpolation theorem, it extends to a
bounded linear map from 
$\Lp^p(B_{\gamma(R)})$ into $\Lp^p(B_R)$ for any $p\in(p_1,p_2)$.
The result then follows as in the proof of \cref{T2.5}.
\end{proof}

\begin{remark}
It is evident from \cref{C2.8}
that if $\nu_x$ has locally compact support and a density $\psi_x\in\Lp^p(\Rd)$
for some $p>\frac{d}{2}$, such that
$x\mapsto\norm{\psi_x}_{\Lp^p(\Rd)}$ is locally bounded,
then the density of an infinitesimal invariant measure is H\"older continuous.
\end{remark}

%
%
%

\section{The HJB equations}\label{S3}

We first discuss the relationship between infinitesimal invariant probability
measures and Foster--Lyapunov equations.
Next, we derive the $\alpha$-discounted HJB equation, and proceed to
study the ergodic HJB equation using the vanishing discount approach.
The treatment is analytical, and we refrain from using any stochastic
representations of solutions.
We state hypothesis \ref{H1} which was discussed in \cref{S1}.

\begin{describe}
\item[\namedlabel{H1}{(H1)}]
For any $v\in\Usm$, and $\alpha>0$, the equation
$\Ag_v u-\alpha u=0$ has no bounded positive solution
$u\in\Sobl^{2,d}(\Rd)$.
\end{describe}

\subsection{On the Foster--Lyapunov equation}

Consider the hypothesis:
\begin{describe}
\item[\namedlabel{H2}{(H2)}] There exist $\Hat{v}\in\Usm$,
a nonnegative $\Lyap\in\Cc^2(\Rd)$, an open ball $\widehat\sB$,
and a positive constant $\kappa_0$ such that
\begin{equation}\label{E-H2}
\Ag_{\Hat{v}} \Lyap(x) \;\le\; \kappa_0\Ind_{\widehat\sB}(x) -  \rc_{\Hat{v}}(x) \quad
\forall\,x\in\Rd\,.
\end{equation}
\end{describe}

On the other hand, $\varrho_*$ is finite if and only if
\begin{describe}
\item[\namedlabel{H3}{(H3)}]
There exist $\Hat{v}\in\Ussm$, and a probability measure $\mu_{\Hat{v}}$
which solves \cref{E-infmeas},
and $\mu_{\Hat{v}}(\rc_{\Hat{v}})=\int\rc_{\Hat{v}}\,\D\mu_{\Hat{v}}<\infty$.
\end{describe}

For continuous diffusions,
equivalence of \ref{H2} and \ref{H3} is a celebrated result of
Has$^{_{^{^\prime}}}\!$minski\u{\i}
\cite{Hasm-80}.
It is pretty straightforward to show, using probabilistic arguments,
that \ref{H2}$\,\Rightarrow\,$\ref{H3}, and this is in fact true
for a large class of Markov processes.
An analytical argument for continuous diffusions
can be found in the work of
Bogachev and R\"ockner \cite{Bogachev-01d}, under the hypothesis
that $\rc_{\Hat{v}}$ is inf-compact.
The argument offered by Has$^{_{^{^\prime}}}\!$minski\u{\i} in the proof that
\ref{H3}$\,\Rightarrow\,$\ref{H2} relies crucially on the Harnack property,
and therefore is not applicable for the jump diffusions considered here.
In the context of general Markov processes,
existence of a solution to \cref{E-H2} is related to the
$f$-regularity of the process. 
For recent work on this, see \cite{Kontoyiannis-16}.

In some sense, \ref{H2} is a very mild assumption, since in any application one would
first need to establish that $\varrho_*$ is finite, and the natural venue
for this is via the Foster--Lyapunov equation in \cref{E-H2}.
A typical example is when $\nu_x$ is translation invariant,
$a$ has sublinear growth, and for some $\theta\in[1,2]$,
$\int_\Rd \abs{y}^\theta\nu(\D{y})<\infty$,
$\rc_{\Hat{v}}$ grows at most as $\abs{x}^{2(\theta-1)}$,
and there exist a positive definite symmetric matrix $S$,
and positive constants $c_0$ and $c_1$
such that $\langle b_{\Hat{v}}(x),Sx\rangle \le c_0 - c_1\abs{x}^\theta$.
Then \cref{E-H2} holds with $\Lyap(x)= \langle x,Sx\rangle^{\nicefrac{\theta}{2}}$.
For other examples, see \cite[Corollary~5.1]{APS17}.

Consider the class of $\nu_x$ that are either translation invariant and
have compact support, or satisfy the hypotheses of \cref{C2.8}, and denote
it by $\mathfrak{N}_0$ for convenience.
For $\nu_x\in\mathfrak{N}_0$, we bridge the gap between \ref{H2} and \ref{H3}
in \cref{T3.7} by establishing
the existence of a solution to the Poisson equation, and thus showing
that \ref{H3}$\,\Rightarrow\,$\ref{H2}, albeit for a function $\Lyap\in\Sobl^{2,p}(\Rd)$.
This however is enough to relax \ref{H2}
in asserting the existence of a solution to the
ergodic HJB for $\nu_x\in\mathfrak{N}_0$ (\cref{T3.8}).
Moroever, the proof of \cref{T3.8} contains an analytical argument which shows
that \ref{H2}$\,\Rightarrow\,$\ref{H3}, provided that $\nu_x\in\mathfrak{N}_0$,
and $\rc_{\Hat{v}}$ is inf-compact.

We need the following simple assertion.

\begin{lemma}\label{L3.1}
Let $\mu_v$ be an infinitesimal invariant measure under $v\in\Ussm$.
Then \cref{E-infmeas} holds for all $\varphi\in\Sobl^{2,p}(\Rd)\cap\Cc_c(\Rd)$, $p>d$.
In addition, if $\varphi\in\Sobl^{2,p}$, $p>d$, is inf-compact,
and such that $\Ag_v \varphi$ is nonpositive
a.e.\ on the complement of some ball $\sB\subset\Rd$, then
$\mu_v\bigl(\abs{\Ag_v \varphi}\bigr)<\infty$.
\end{lemma}

\begin{proof}
In the interest of simplicity,
we drop the explicit dependence on $v$ in the notation.
Suppose $\varphi\in\Sobl^{2,p}(\Rd)\cap\Cc_c(\Rd)$, $p>d$.
Let $\rho$ be a symmetric non-negative mollifier supported on the unit ball
centered at the origin, and for $\epsilon>0$, let $\rho_{\epsilon}(x)\df
r^{-d}\rho (\frac{x}{\epsilon})$, and
$\varphi_\epsilon\df \rho_\epsilon*\varphi$, where `$*$' denotes convolution.
Then, $\mu(\Ag \varphi_\epsilon)=0$ by \cref{E-infmeas}.
Since $\partial_{ij}\varphi_\epsilon$ converges to $\partial_{ij}\varphi$ as
$\epsilon\searrow0$ in
$\Lp^p(B_R)$ for any $p>1$ and $R>0$, and since $\mu$ has a density
in $\Lpl^p(\Rd)$ for $p<\frac{d}{d-2}$ by \cref{T2.5}, it follows by
H\"older's inequality that
$\int_\Rd \abs{a^{ij}}
\abs{\partial_{ij}\varphi-\partial_{ij}\varphi_\epsilon}\,\D{\mu}\to0$
as $\epsilon\searrow0$.
Also, since $\partial_i \varphi-\partial_i \varphi_\epsilon$ converges uniformly to $0$,
and in view of \cref{A1.1}~(b) and (c), we obtain
$\mu(\Hat{b}^i\partial_{i}\varphi_\epsilon)
\to\mu(\Hat{b}^i\partial_{i}\varphi)$, and
$\mu(\cI\varphi_\epsilon)\to\mu(\cI \varphi)$ as $\epsilon\searrow0$.
This shows that $\mu(\Ag \varphi)=0$.

We now turn to the second statement of the lemma.
Let $\chi$ be a concave $\Cc^2(\Rd)$ function such that
$\chi(x)= x$ for $x\le0$, and $\chi(x) = 1$ for $x\ge 1$.
Then $\chi'$ and $-\chi''$ are nonnegative on $(0,1)$.
Define $\chi_{R}^{}(x) \df R + \chi(x-R)$ for $R>0$,
and observe that $\chi_R^{}(\varphi)-R-1$ is compactly supported by construction.
We have
\begin{equation}\label{PL3.1A}
\Ag \chi_R^{} (\varphi) \;=\; \chi_R^{\prime}(\varphi)\,\Ag \varphi
+ \chi_R^{\prime\prime}(\varphi)\,\langle\grad \varphi, a\grad \varphi\rangle
- \bigl(\chi_R^{\prime}(\varphi)\cI \varphi - \cI \chi_R(\varphi)\bigr)\,.
\end{equation}
Note that the second and third terms on the right hand side of \cref{PL3.1A}
are nonpositive.
Thus, selecting $R$ sufficiently large so that
$\Ag \varphi$ is nonpositive on the complement of $B_R$,
and integrating with respect to
$\mu$, we first obtain $\mu\bigl((\Ag \varphi)^+\bigr)<\infty$, and using
this together with \cref{PL3.1A} the result follows.
\end{proof}

\subsection{The \texorpdfstring{$\alpha$}{}-discounted HJB equation}

We have the following theorem.

\begin{theorem}\label{T3.2}
Assume \textup{\ref{H1}}--\textup{\ref{H2}}.
For any $\alpha\in(0,1)$, there exists a minimal nonnegative solution
$V_\alpha\in\Sobl^{2,p}(\RR^d)$, $p>1$, to the HJB equation
\begin{equation}\label{ET3.2A}
\min_{z\in\Act}\;\bigl[\Ag_{z}\,V_\alpha(x) + \rc(x,z)\bigr] \;=\; \alpha V_\alpha(x)\,.
\end{equation}
Moreover, $\inf_{\Rd}\,\alpha V_\alpha\le\varrho_*$, and this
infimum is attained in the set
$$\varGamma_\circ\;\df\;\Bigl\{x\in\Rd\,\colon \sup_{z\in\Act}\,\rc(x,z)\le \varrho_*
\Bigr\}\,.$$
\end{theorem}


\begin{proof}
Establishing the existence of a solution is quite standard.
One starts by exhibiting a solution
$\psi_{\alpha,R}\in\Sob^{2,p}(B_R)\cap\Cc(\Rd)$
to the Dirichlet problem
\begin{equation}\label{PT3.2A}
\begin{cases}
\min_{z\in\Act}\;\bigl[\Ag_z\psi_{\alpha,R}(x)
+ \rc(x,z)\bigr]\;=\;\alpha\psi_{\alpha,R}(x)
& x\in B_R\,,\\
\psi_{\alpha,R}(x)\;=\;0& x\in B_R^c\,,
\end{cases}
\end{equation}
for any $\alpha\in(0,1)$ and $R>0$.

We use \cref{D2.4} to write $\Ag=\widetilde\Lg + \widetilde\cI$.
Applying the well-known interior estimate in \cite[Theorem~9.11]{GilTru},
for any fixed $r>0$, we obtain
\begin{equation*}
\bnorm{\psi_{\alpha,R}}_{\Sob^{2,p}(B_{r})}\le
C\Bigl(\bnorm{\psi_{\alpha,R}}_{\Lp^p(B_{2r})}+
\bnorm{\rc_{v_\alpha}
+\widetilde\cI\,\psi_{\alpha,R}}_{\Lp^p(B_{2r})}\Bigr)
\end{equation*}
for some constant $C= C(r,p)$.
Here, $v_\alpha$ is a measurable selector from the minimizer of
the $\alpha$-discounted HJB in \eqref{ET3.2A}.
Using the comparison principle and \ref{H2}, it is straightforward to show that
$\psi_{\alpha,R} \le \frac{\kappa_\circ}{\alpha} + \Lyap$ on $\Rd$.
Thus
$\{\psi_{\alpha,R}\}$ is bounded in $\Sob^{2,p}(B_{r})$, uniformly in $R$.
We then take limits as $R\to\infty$ to obtain a function
$V_\alpha\in\Sobl^{2,p}(\RR^d)$ which solves \cref{ET3.2A}.

Let $m_\alpha\df \inf_\Rd\,V_\alpha$.
We claim that $\alpha m_\alpha\le\varrho_*$.
Suppose on the contrary that $\alpha m_\alpha>\varrho_*$.
Let $v\in\Ussm$.
Recall the function $\chi$ in the proof of \cref{L3.1}, and let
$\Tilde\chi(x) \df  - \chi(\frac{\varrho_*}{2}+1-x)$.
Note that $\Tilde\chi^{\prime\prime}\ge0$,
and $\Tilde\chi^{\prime}(\psi_{\alpha,R})\cI \psi_{\alpha,R}
- \cI \Tilde\chi(\psi_{\alpha,R})\le 0$.
Thus, using \cref{PT3.2A} and repeating the calculation in \cref{PL3.1A}
we obtain
\begin{equation*}
\Ag_v \Tilde\chi(\psi_{\alpha,R}) \;\ge\;
\Tilde\chi^{\prime}(\psi_{\alpha,R})\,\Ag_v \psi_{\alpha,R}
\;\ge\;
\Tilde\chi^{\prime}(\psi_{\alpha,R})\,\bigl(\alpha \psi_{\alpha,R} - \rc_v\bigr)\,.
\end{equation*}
It is clear that $\Tilde\chi(\psi_{\alpha,R})\in\Sobl^{2,p}(\Rd)\cap\Cc_c(\Rd)$,
for any $p>1$.
Hence, integrating with respect to $\mu_v$, applying \cref{L3.1}, and
taking limits as $R\to\infty$, using monotone convergence, we obtain
$\alpha m_\alpha\le\mu_v(\alpha V_\alpha) \le \mu_v(\rc_v)$.
Taking the infimum over $v\in\Ussm$ contradicts the hypothesis
that $\alpha m_\alpha>\varrho_*$, and thus proves the claim.

Recall the definition $\epsilon_\circ$ in \cref{S2.2}.
Let $\Tilde{v}\in\Usm$ be a measurable selector from the minimizer
of \cref{PT3.2A} and consider the Dirichlet problem
\begin{equation}\label{PT3.2C}
\begin{cases}
\Ag_{\Tilde{v}}\Tilde\psi_{\alpha,R}(x)
+ \rc_{\Tilde{v}}(x)\;=\;\alpha\Tilde\psi_{\alpha,R}(x)
& x\in B_R\,,\\
\Tilde\psi_{\alpha,R}(x)\;=\;\alpha^{-1}(\varrho_*+\epsilon_\circ)& x\in B_R^c\,,
\end{cases}
\end{equation}
for $\alpha\in(0,1)$ and $R>0$.
Arguing as in the derivation of \cref{PT3.2A},
it follows that $\Tilde\psi_{\alpha,R}$ converges,
as $R\to\infty$, to some
$\Tilde{V}_{\alpha}\in\Sobl^{2,p}(\RR^d)$ which solves
$\Ag_{\Tilde{v}}\Tilde{V}_{\alpha}
+ \rc_{\Tilde{v}}(x)\;=\;\alpha\Tilde{V}_{\alpha}$ on $\Rd$.
It is clear that $u=\Tilde{V}_{\alpha}-V_\alpha$ is nonnegative and
bounded. Since $\Ag_{\Tilde{v}}u -\alpha u=0$ on $\Rd$, it follows by
\ref{H1} that $u$ cannot be strictly positive, and, in turn, by the strong
maximum principle it has to be identically zero.
Thus, given $\epsilon<\epsilon_\circ$ there exists $R_\epsilon$ such
that $\min_{B_R}\,\alpha \Tilde\psi_{\alpha,R}<\varrho_*+\epsilon$
for all $R>R_\epsilon$.
It follows by \cref{PT3.2C} that $\Tilde\psi_{\alpha,R}$
attains its minimum in the set
$\varGamma_\epsilon\df\{x\in\Rd\,\colon \rc_{\Tilde{v}}(x)\le \varrho_*+\epsilon\}$
for all $R>R_\epsilon$, and therefore, the same applies to $\Tilde{V}_{\alpha}$.
Since $\epsilon>0$ is arbitrary, we conclude that
$\Tilde{V}_{\alpha}$ attains its infimum in the set
$\{x\in\Rd\,\colon \rc_{\Tilde{v}}(x)\le \varrho_*\}\subset\varGamma_\circ$,
and this completes the proof.
\end{proof}

\subsection{The ergodic HJB equation}

We start with the main convergence result of the paper which
establishes solutions to the ergodic HJB via the vanishing discount method.
To guide the reader, the technique of the proof consists of writing the
operator in the form $\widetilde\Lg + \widetilde\cI$, and obtaining
estimates for supersolutions of the local operator $\widetilde\Lg$ using the results
in \cite[Corollary~2.2]{AA-Harnack}.

\begin{theorem}\label{T3.3}
Grant the hypotheses of \cref{T3.2}, and
let $V_\alpha$, $\alpha\in(0,1)$, be the family of solutions in that theorem.
Then, as $\alpha\searrow0$,
$V_\alpha-V_\alpha(0)$ converges in $\Cc^{1,r}(\overline B_R)$
for any $r\in(0,1)$ and $R>0$, to a function $V\in\Sobl^{2,p}(\Rd)$
for any $p>1$, 
which is bounded from below in $\Rd$ and solves
\begin{equation}\label{ET3.3A}
\min_{z\in\Act}\;\bigl[\Ag_{z}\,V(x) + \rc(x,z)\bigr] \;=\; \varrho\,,
\end{equation}
with $\varrho=\varrho_*$.
Also $\alpha V_\alpha(x)\to \varrho_*$ uniformly on compact sets.
In addition, the solution of \cref{ET3.3A} with $\varrho=\varrho_*$ is unique
in the class of functions  $V\in\Sobl^{2,d}(\Rd)$, satisfying $V(0)=0$,
which are bounded from below in $\Rd$.
For $\varrho<\varrho_*$, there is no such solution.
\end{theorem}

\begin{proof}
Recall the definitions of $\sB_\circ$ and $\epsilon_\circ$ in \cref{S2.2}.
Fix an arbitrary ball $\sB\subset\Rd$ such that $\sB_\circ\subset\sB$.
Since $\Lyap$ and $V_\alpha$
are a supersolution and subsolution of $\Ag_{\Hat{v}}u -\alpha u = -\rc_{\Hat{v}}$
on $\sB^c$ by \cref{E-H2}, respectively, it follows that the solution $V_\alpha$
of \cref{ET3.2A} satisfies
\begin{equation}\label{PT3.3A}
V_\alpha(x) \;\le\; \sup_{\sB}\;V_\alpha + \Lyap(x)\qquad \forall\,x\in\Rd\,.
\end{equation}
By \cref{T3.2} we have
$\inf_\Rd\, V_\alpha = \min_{\sB_\circ}\, V_\alpha$ for all $\alpha\in(0,1)$.
For each $\alpha\in(0,1)$, we
fix some point $\Hat{x}_\alpha\in\Argmin V_\alpha\subset\sB_\circ$.
Consider the function $\varphi_\alpha\df V_\alpha - V_\alpha(\Hat{x}_\alpha)$.
Then \cref{PT3.3A} implies that
\begin{equation}\label{PT3.3B}
\varphi_\alpha(x) \;\le\; \norm{\varphi_\alpha}_{\Lp^\infty(\sB_\circ)}
+ \Lyap(x)\qquad \forall\,x\in\Rd\,.
\end{equation}
We  have
\begin{equation*}
\min_{z\in\Act}\;\bigl[\Ag_z\varphi_\alpha(x) -\alpha\varphi_\alpha(x)
+ \rc(x,z)\bigr] \;=\; \alpha V_\alpha(\Hat{x}_\alpha)\;\le\;\varrho_*\,,
\end{equation*}
where the last inequality follows by \cref{T3.2}.
We claim that for each $R>0$ there exists a constant
$\kappa_R$ such that
\begin{equation}\label{PT3.3C}
\norm{\varphi_\alpha}_{\Lp^\infty(B_R)}\;\le\; \kappa_R\qquad\forall\,\alpha\in(0,1)\,.
\end{equation}
To prove the claim, let $\sB\equiv B_R$, and $D_1$, $D_2$ be balls satisfying
$\sB\Subset D_1\Subset D_2$.
Recall \cref{D2.4}.
For $p>0$, let $\norm{u}_{p;Q}\df\bigl(\int_Q u(x)\,\D{x}\bigr)^{\nicefrac{1}{p}}$.
Of course, this is not a norm unless $p\ge1$, so there is a slight abuse
of notation involved in this definition.
Since $\Lyap\in\Cc^{2}(\Rd)$, hypothesis \ref{H2} implies that
$\widetilde\cI\Lyap\in\Lpl^{\infty}(\Rd)$, and the same
of course holds for $\varphi_\alpha$ by \cref{PT3.3B}.
By the local maximum principle \cite[Theorem~9.20]{GilTru}, for any $p>0$, there exists
a constant $\Tilde{C}_{1}(p)>0$ such that
\begin{equation*}
\norm{\varphi_\alpha}_{\Lp^\infty(\sB)}
\;\le\;\Tilde{C}_{1}(p)\bigl(\norm{\varphi_\alpha}_{p;D_1}
+ \norm{\widetilde\cI\,\varphi_\alpha}_{\Lp^d(D_1)}
+\norm{\rc_{v_\alpha}}_{\Lp^d(D_1)}\bigr)\,,
\end{equation*}
and by the supersolution estimate \cite[Theorem~9.22]{GilTru},
and since $\varphi_\alpha$ is nonnegative,
there exist some
$p>0$ and $\Tilde{C}_{2}>0$ such that
$\norm{\varphi_\alpha}_{p;D_1} \;\le\; \Tilde{C}_{2}\,\varrho_*\,
\abs{D_2}^{\nicefrac{1}{d}}$.
Combining these inequalities, we obtain
\begin{equation}\label{PT3.3D}
\norm{\varphi_\alpha}_{\Lp^\infty(\sB)}  \;\le\;
\Tilde{C}_{1}(p)\bigl(\Tilde{C}_{2}\,\varrho_*\,\abs{D_2}^{\nicefrac{1}{d}}
+\norm{\rc_{v_\alpha}}_{\Lp^d(D_1)}\bigr)
+ \Tilde{C}_{1}(p) \norm{\widetilde\cI\,\varphi_\alpha}_{\Lp^d(D_1)}\,.
\end{equation} 
Denote the first term on the right hand of \eqref{PT3.3D} by $\kappa_1$. 
By \cref{PT3.3B,PT3.3D} we have
\begin{align*}
\norm{\varphi_\alpha}_{\Lp^\infty(D_2)} \;&\le\;
\norm{\Lyap}_{\Lp^\infty(D_2)} + \norm{\varphi_\alpha}_{\Lp^\infty(\sB)}\\
&\le\; \kappa_1 + \norm{\Lyap}_{\Lp^\infty(D_2)} +
\Tilde{C}_{1}(p)\,\norm{\widetilde\cI\,\varphi_\alpha}_{\Lp^d(D_1)}\,.
\end{align*}
This implies that, either
$\norm{\varphi_\alpha}_{\Lp^\infty(D_2)}
\le 2\bigl(\kappa_1 +\norm{\Lyap}_{\Lp^\infty(D_2)}\bigr)$,
in which case \eqref{PT3.3C} holds with this bound, or
\begin{equation}\label{PT3.3E}
\norm{\varphi_\alpha}_{\Lp^\infty(D_2)}\;\le\;
2\Tilde{C}_{1}(p)\,\norm{\widetilde\cI\,\varphi_\alpha}_{\Lp^d(D_1)}\,.
\end{equation}
If \cref{PT3.3E} holds, then we write 
$\widetilde\cI\,\varphi_\alpha= \widetilde\cI(\Ind_{D_2}\varphi_\alpha)
+\widetilde\cI(\Ind_{D_2^c}\varphi_\alpha)$,
and use the estimate
\begin{equation*}
\widetilde\cI(\Ind_{D_2^c}\varphi_\alpha)(x)\;\le\;
\norm{\varphi_\alpha}_{\Lp^\infty(\sB)}\,
\biggl(\sup_{x\in D_1}\,\nu_x(D_2^c)\biggr)
+\widetilde\cI(\Ind_{D_2^c}\Lyap)
\qquad\forall\,x\in D_1\,,
\end{equation*}
which holds by \cref{PT3.3B}, together with  \cref{PT3.3D,PT3.3E},
to obtain
\begin{align}\label{PT3.3F}
\norm{\widetilde\cI\,\varphi_\alpha}_{\Lp^\infty(D_1)}&\;\le\;
2\Tilde{C}_{1}(p)\,\norm{\bm\nu}_{\Lp^\infty(D_1)}\,
\norm{\widetilde\cI\,\varphi_\alpha}_{\Lp^d(D_1)}\\
&\mspace{100mu}
+\kappa_1\norm{\bm\nu}_{\Lp^\infty(D_1)}
+\norm{\widetilde\cI(\Ind_{D_2^c}\Lyap)}_{\Lp^\infty(D_1)}\,.\nonumber
\end{align}
We distinguish two cases from \cref{PT3.3F}:

\noindent
\textbf{Case 1.} Suppose that
\begin{equation}\label{PT3.3G}
\norm{\widetilde\cI\,\varphi_\alpha}_{\Lp^\infty(D_1)}\;\le\;
4\Tilde{C}_{1}(p)\,\norm{\bm\nu}_{\Lp^\infty(D_1)}\,
\norm{\widetilde\cI\,\varphi_\alpha}_{\Lp^d(D_1)}\,.
\end{equation}
Let $\psi_\alpha$ be the solution of the Dirichlet problem
\begin{equation*}
\widetilde\Lg_{v_\alpha} \psi_\alpha - \alpha \psi_\alpha \,=\,
-\widetilde\cI\, \varphi_\alpha
\quad \text{in\ } D_1\,,\qquad
\text{and} \quad\psi_\alpha = \varphi_\alpha\quad\text{on\ }\partial D_1\,.
\end{equation*}
Then $\psi_\alpha$ is nonnegative in $D_1$
by the strong maximum principle, and thus \cref{PT3.3G} together
with \cite[Corollary~2.2]{AA-Harnack},
implies that for some constant $C_{\mathsf H}$ we have
\begin{equation}\label{PT3.3H}
\psi_\alpha(x) \;\le\; C_{\mathsf H}\,\psi_\alpha(\Hat{x}_\alpha)
\qquad \forall\,x\in \sB\,,\quad\forall\,\alpha\in(0,1)\,.
\end{equation}
On the other hand, $\varphi_\alpha-\psi_\alpha$ satisfies
\begin{equation}\label{PT3.3I}
\widetilde\Lg_{v_\alpha} (\varphi_\alpha-\psi_\alpha)
-\alpha (\varphi_\alpha-\psi_\alpha)\;=\;
\alpha V_\alpha(\Hat{x}_\alpha) - \rc_{v_\alpha}
\quad \text{in\ } D_1\,,
\end{equation}
and $\varphi_\alpha-\psi_\alpha = 0$ on $\partial D_1$.
Thus, by the ABP weak maximum
principle \cite[Theorem~9.1]{GilTru},
and since $\alpha V_\alpha(\Hat{x}_\alpha)\le\varrho_*$,
we obtain from \cref{PT3.3I} that
\begin{equation}\label{PT3.3J}
\norm{\varphi_\alpha-\psi_\alpha}_{\Lp^\infty(D_1)}
\;\le\; C_\circ\qquad\forall\,\alpha\in(0,1)\,,
\end{equation}
for some constant $C_\circ$.
\Cref{PT3.3J} implies that $\psi_\alpha(\Hat{x}_\alpha)\le C_\circ$.
Combining \cref{PT3.3H,PT3.3J} in the standard manner, we obtain
\begin{align}\label{PT3.3K}
\varphi_\alpha(x) &\;\le\; \norm{\varphi_\alpha-\psi_\alpha}_{\Lp^\infty(D_1)}
+\psi_\alpha(x)\\
&\;\le\; C_\circ + C_{\mathsf H}\,\psi_\alpha(\Hat{x}_\alpha)
\;\le\; C_\circ(1+C_{\mathsf H})\qquad\forall\,x\in\sB\,,\quad\forall\,\alpha\in(0,1)\,.
\nonumber
\end{align}

\noindent
\textbf{Case 2.} Suppose that
\begin{equation*}
\norm{\widetilde\cI\,\varphi_\alpha}_{\Lp^\infty(D_1)}\;\le\;
2\kappa_1\norm{\bm\nu}_{\Lp^\infty(D_1)}\,
+2\,\norm{\widetilde\cI(\Ind_{D_2^c}\Lyap)}_{\Lp^\infty(D_1)}\,.
\end{equation*}
In this case, we consider the solution $\Tilde\psi_\alpha$
of the Dirichlet problem
\begin{equation*}
\widetilde\Lg_{v_\alpha} \Tilde\psi_\alpha -\alpha \Tilde\psi_\alpha \,=\, 0
\quad \text{in\ } D_1\,,\qquad
\text{and} \quad\Tilde\psi_\alpha = \varphi_\alpha\quad\text{on\ } \partial D_1\,.
\end{equation*}
We have
$\Tilde\psi_\alpha(x) \le \Tilde C_{\mathsf H}\,\Tilde\psi_\alpha(\Hat{x}_\alpha)$
for all $x\in \sB$ and $\alpha\in(0,1)$,
for some constant $\Tilde C_{\mathsf H}$.
Also,
\begin{equation}\label{PT3.3M}
\widetilde\Lg_{v_\alpha} (\varphi_\alpha-\Tilde\psi_\alpha)
-\alpha (\varphi_\alpha-\psi_\alpha)
\,=\, -\widetilde\cI\, \varphi_\alpha + \alpha V_\alpha(\Hat{x}_\alpha) - \rc_{v_\alpha}
\quad \text{in\ } D_1\,,
\end{equation}
and $\varphi_\alpha-\Tilde\psi_\alpha = 0$ on $\partial D_1$.
By the ABP weak maximum principle, we obtain from
\cref{PT3.3M} that
$\norm{\varphi_\alpha-\Tilde\psi_\alpha}_{\Lp^\infty(D_1)}\le\Tilde C_\circ$
for all $\alpha\in(0,1)$ and
for some constant $\Tilde C_\circ$.
Thus again we obtain \cref{PT3.3K} with constants $\Tilde C_\circ$
and $\Tilde C_{\mathsf H}$.
This establishes \cref{PT3.3C}.

It follows by \cref{PT3.3C} that
$\overline{V}_\alpha\df V_\alpha - V_\alpha(0)=\varphi_\alpha(x)-
\varphi_\alpha(0)$ is locally bounded,
uniformly in $\alpha\in(0,1)$.
The same applies to $\widetilde\cI\, \overline{V}_\alpha$ by \cref{PT3.3B} and \ref{H2}.
Note that
\begin{equation*}
\widetilde\Lg_{v_\alpha} \overline{V}_\alpha 
-\alpha \overline{V}_\alpha\,=\, \alpha V_\alpha(0) - \rc_{v_\alpha}
-\widetilde\cI\, \overline{V}_\alpha
\quad \text{on\ } \Rd\,.
\end{equation*}
Thus, by the interior estimate in \cite[Theorem~9.11]{GilTru}, there exists
a constant $C= C(R,p)$ such that
\begin{equation*}
\bnorm{\overline{V}_\alpha}_{\Sob^{2,p}(B_{R})}\;\le\;
C\Bigl(\bnorm{\overline{V}_\alpha}_{\Lp^p(B_{2R})}+
\bnorm{\alpha V_\alpha(0) - \rc_{v_\alpha}
-\widetilde\cI\, \overline{V}_\alpha}_{\Lp^p(B_{2R})}\Bigr)\,.
\end{equation*}
Hence $\{\overline{V}_\alpha\}$ is bounded in $\Sob^{2,p}(B_{R})$ for any
$R>0$.
A standard argument then shows that
given any sequence
$\alpha_n\searrow0$, $\{\overline{V}_{\alpha_n}\}$
contains a subsequence which converges in $\Cc^{1,r}(\overline{B}_R)$
for any $r<1-\frac{d}{p}$ (see, e.g., Lemma~3.5.4 in \cite{ABG12}).
Taking limits in
\begin{equation}\label{ET3.3N}
\min_{z\in\Act}\;\bigl[\Ag_z\overline{V}_\alpha(x) -\alpha\overline{V}_\alpha(x)
+ \rc(x,z)\bigr] \;=\; \alpha V_\alpha(0)
\end{equation}
along this subsequence we obtain \cref{ET3.3A}, as claimed in the statement
of the theorem, for some $\varrho\in\RR$.
Since $\limsup_{\alpha\searrow0}\;\alpha V_\alpha(\Hat{x}_\alpha)\le\varrho_*$,
we have $\varrho\le\varrho_*$.
On the other hand,
from the theory of infinite dimensional linear programming
\cite{AndNash} it is well known that the value of the dual problem cannot be
smaller than the value of the primal, hence $\varrho\ge\varrho_*$, and we have
equality (see also Section~4 in \cite{BhBo}).

Suppose now that $\widetilde{V}\in\Sobl^{2,d}(\Rd)$
is bounded from below in $\Rd$, and satisfies
\begin{equation}\label{ET3.3O}
\min_{z\in\Act}\;\bigl[\Ag_{z}\,\widetilde{V}(x) + \rc(x,z)\bigr]
\;=\; \varrho_*\,.
\end{equation}
Let $\Tilde{v}\in\Usm$ be an a.e.\ measurable selector from the minimizer
of \cref{ET3.3O}.
Define $\widetilde{V}^\epsilon\df (1+\epsilon)\widetilde{V}$, $\epsilon>0$.
Arguing as in the derivation of  \cref{PT3.3B}, it is clear that
this equation holds with $\Lyap$ replaced by $\widetilde{V}^\epsilon$.
Translate $\widetilde{V}^\epsilon$ by an additive constant until it touches
$\varphi_\alpha$ at some point from above.
Since
\begin{equation*}
\Ag_{\Tilde{v}}(\widetilde{V}^\epsilon-\varphi_\alpha)
- \alpha (\widetilde{V}^\epsilon-\varphi_\alpha)
\;\le\; (1+\epsilon)\varrho_* - \alpha\varphi_\alpha(\Hat{x}_\alpha)-\rc_{\Tilde{v}}\,,
\end{equation*}
taking first limits as $\alpha\searrow0$,
and then as $\epsilon\searrow0$, we obtain
$\Ag_{\Tilde{v}}(\widetilde{V}-V)\le0$,
and conclude that $\widetilde{V}=V$ by the strong maximum principle.

It is evident from the uniqueness of the solution, that
the limit of \cref{ET3.3N} is independent of the subsequence $\alpha_n\searrow0$
chosen.
It is also clear that $\alpha V_\alpha(x)\to\varrho_*$ as $\alpha\searrow0$,
uniformly on compact sets.
This completes the proof.
\end{proof}

\begin{remark}\label{R3.4}
If $\nu_x$ is translation invariant and has compact support,
and $\rc$ and $b$ are locally H\"older continuous in $x$,
then $\widetilde\cI V$ is locally H\"older continuous,
and thus the solution $V$ in \cref{T3.3} is in
$\Cc^{2,r}(\Rd)$ for some $r\in(0,1)$
by elliptic regularity
\cite[Theorem~9.19]{GilTru}.
\end{remark}

\subsubsection{Verification of optimality}
We start with the following theorem.

\begin{theorem}\label{T3.5}
Assume the hypotheses of \cref{T3.3}.
If $v\in\Ussm$ is optimal, then it satisfies
\begin{equation}\label{ET3.5A}
b^i_{v}(x)\,\partial_i V(x)+\rc_{v}(x)
\;=\;  \inf_{z\in\Act}\;\bigl[b^i(x,z)
\partial_i V(x)+\rc(x,z)\bigr]\qquad\text{a.e.\ } x\in\Rd\,.
\end{equation}
In addition, provided $V$ is inf-compact,
any stable $v\in\Ussm$ which satisfies
\cref{ET3.5A} is necessarily optimal.
\end{theorem}

\begin{proof}
Suppose not. Then there exists some ball $\sB$ such that
\begin{equation}\label{PT3.5A}
h(x) \;\df\; \Bigl(b^i_{v}(x)\,\partial_i V(x)+\rc_{v}(x)
-\inf_{z\in\Act}\;\bigl[b^i(x,z)
\partial_i V(x)+\rc(x,z)\bigr]\Bigr)\,\Ind_{\sB}(x)
\end{equation}
is a nontrivial nonnegative function.
Since $\partial_i V_\alpha$ converges uniformly to $\partial_i V$
as $\alpha\searrow0$ on compact sets by \cref{T3.3}, it follows
that if we define $h_\alpha$ as the right hand side of \cref{PT3.5A}, but
with $V$ replaced by $V_\alpha$, then
$h-h_\alpha$ converges to $0$ a.e.\ in $\sB$, and also
$\mu_v(\abs{h-h_\alpha})\to0$ as $\alpha\searrow0$, since $\mu_v$
has a density in $\Lpl^p(\Rd)$ for some $p>1$.
We have $\Ag_v V_\alpha\ge \alpha V_\alpha + h_\alpha - \rc_v$ a.e.\ on $\Rd$
by the definition of $h_\alpha$.
With $\psi_{\alpha,R}$ the solution in \cref{PT3.2A}, and
$m_\alpha=\inf_\Rd\, V_\alpha$, and define
$\Breve\psi_{\alpha,R}\df \psi_{\alpha,R}- m_\alpha$.
Repeating the above argument, there exists $\Breve{h}_{\alpha,R}$ supported
on $\sB$ such that $\mu_v(\abs{\Breve{h}_{\alpha,R}-h_\alpha})\to0$ as $R\to\infty$,
and
$\Ag_v\Breve\psi_{\alpha,R}
\ge \alpha\psi_{\alpha,R} + \Breve{h}_{\alpha,R}- \rc_v$.
We apply the function $\Breve\chi(x) \df - \chi(\frac{\varrho_*}{2}+1-x)$,
with $\chi$ as defined in the proof of \cref{L3.1}, and repeat the
argument in \cref{T3.2}, also letting $R\to\infty$, to obtain
$\mu_v(\rc_v) \;\ge\; \mu_v(\alpha V_\alpha) + \mu_v(h_\alpha)$.
By the proof of \cref{T3.3} $\inf_\Rd\,\alpha V_\alpha\to\varrho_*$
as $\alpha\searrow0$.
Thus, taking limits as $\alpha\searrow0$, we obtain $\mu_v(h)\le0$, and since
$\mu_v$ has everywhere positive density, this implies $h=0$ a.e.

The second assertion of the theorem is easily
established by the argument in the proof of \cref{L3.1}, using
the function $\chi^{}_R$.
\end{proof}

\begin{remark}
If we impose the additional assumption that the coefficients
$a$ and $b$ have at most affine growth, and that $\nu_x(B_R-x)$ vanishes
as $\abs{x}\to\infty$, for any ball $B_R$, then it is standard to show that
the solution $V$ in \cref{T3.3} is inf-compact, so that the second
assertion of \cref{T3.5} applies.
However, this leaves open the question whether a $v\in\Usm$
that satisfies \cref{ET3.5A} is necessarily stable.
We provide a partial answer to this in \cref{T3.8} below.
\end{remark}

Recall \cref{D2.7}.
We impose additional assumptions on $\nu_x$ to
establish existence of solutions to the Poisson equation.

\begin{theorem}\label{T3.7}
We assume \textup{\ref{H1}} and one of the following:
\begin{itemize}
\item[\textup{(}a\/\textup{)}]
$\nu=\nu_x$ is translation invariant and has compact support.
\item[\textup{(}b\/\textup{)}]
$\nu_x$ has locally compact support and satisfies the hypotheses
of \cref{C2.8}.
\end{itemize}
Let $\Hat{v}\in\Ussm$ be such that
$\rc_{\Hat{v}}$ is coercive relative to $\varrho_{\Hat{v}}$.
Then, up to an additive constant, there exists a unique
$\widehat{V}\in\Sobl^{2,d}(\Rd)$
which is bounded
from below in $\Rd$, and satisfies
\begin{equation}\label{ET3.7A}
\Ag_{\Hat{v}}\,\widehat{V}(x) + \rc_{\Hat{v}}(x) \;=\; \beta
\qquad\forall\,x\in\Rd\,,
\end{equation}
for some $\beta= \varrho_{\Hat{v}}$.
For $\beta<\varrho_{\Hat{v}}$, there is no such solution.
\end{theorem}

\begin{proof}
For $n\in\NN$, let $\rc^n= n\wedge\rc$ denote the $n$-truncation of the running
cost.
It is clear that $\rc^n$ is coercive relative to $\varrho_{\Hat{v}}$
for all $n>\varrho_{\Hat{v}}$.
Consider the $\alpha$-discounted problem in \cref{T3.2}.
The Dirichlet problem in \cref{PT3.2A}
is now a linear problem, and we let $\Hat{\psi}^n_{\alpha,R}$ denote the
corresponding solution.
It is clear that
$\norm{\Hat{\psi}^n_{\alpha,R}}_{\Lp^\infty(\Rd)}\le\frac{n}{\alpha}$,
and this is inherited by the function
$\widehat{V}^n_\alpha$ at the limit $R\to\infty$.
Thus, by the proof of \cref{T3.2}, $\widehat{V}^n_\alpha$ is in $\Sobl^{2,p}(\Rd)$
for any $p\ge1$, and satisfies
$\Ag_{\Hat{v}}\,\widehat{V}^n_\alpha + \rc_{\Hat{v}}^n = \alpha \widehat{V}^n_\alpha$.
Repeating the argument in the proof of \cref{T3.3},
the infimum of $\widehat{V}^n_\alpha$ over $\Rd$ is attained in a ball
$\sB_\circ$ as defined in \cref{S2.2} (relative to $\varrho_{\Hat{v}}$),
and if $\Hat{x}_\alpha^n\in\sB_\circ$ denotes a point where the infimum is attained,
then
$\alpha \widehat{V}^n_\alpha(\Hat{x}_\alpha^n)\le \varrho_{\Hat{v}}$.
With $\varphi_\alpha^n \df \widehat{V}^n_\alpha
- \widehat{V}^n_\alpha(\Hat{x}_\alpha^n)$, we write
the equation as
\begin{align}\label{PT3.7A}
\widetilde\Lg_{\Hat{v}} \varphi_\alpha^n(x) - \alpha\varphi_\alpha^n(x)
&\;=\; \alpha \widehat{V}^n_\alpha(\Hat{x}_\alpha^n) - \rc^n_{\Hat{v}}(x)
-\widetilde\cI\varphi_\alpha^n(x)\\
&\;\le\; \varrho_{\Hat{v}} -\rc^n(x)-\widetilde\cI\varphi_\alpha^n(x)
\qquad\text{a.e.\ }x\in \Rd\,.\nonumber
\end{align}
We express \cref{PT3.7A} in divergence form as
\begin{equation*}
\partial_j \bigl(a^{ij}\partial_i \varphi_\alpha^n\bigr)
+ \bigl(\Hat{b}^i-\partial_i a^{ij}\bigr)\partial_i \varphi_\alpha^n
-\bm\nu \varphi_\alpha^n
 \;\le\; \varrho_{\Hat{v}} - \rc_{\Hat{v}} -\widetilde\cI \varphi_\alpha^n\,,
\end{equation*}
and apply
\cite[Theorem~8.18]{GilTru} to obtain 
$\bnorm{\varphi_\alpha^n}_{\Lp^p(B_{2R}(x_0))}
\le \varrho_{\Hat{v}}\,\kappa_{p,R}$
for some constant $\kappa_{p,R}$, for any $p\in\bigl(1,\frac{d}{d-2}\bigr)$.
Therefore,
$\inf_{B_{2R}(x_0)\setminus B_{R}(x_0)}\, \varphi_\alpha^n$ is bounded over
$\alpha\in(0,1)$ and $n\ge\varrho_{\Hat{v}}$.
Thus, we can select some $x'_0\in B_{2R}(x_0)\setminus B_{R}(x_0)$
satisfying $\sup_n \varphi_\alpha^n(x_0')<\infty$, and repeat the procedure to
show by induction 
that $\varphi_\alpha^n$ is locally bounded in $\Lp^p$
for any $p\in\bigl(1,\frac{d}{d-2}\bigr)$, uniformly
over $\alpha\in(0,1)$ and $n\ge\varrho_{\Hat{v}}$.

Next, we apply successively the Calder\'on--Zygmund estimate
\cite[Theorem~9.11]{GilTru}
to the non-divergence form of the
equation in \cref{PT3.7A} which states that
\begin{equation*}
\bnorm{\varphi_\alpha^n}_{\Sob^{2,p}(B_{R})}\;\le\;
C\Bigl(\bnorm{\varphi_\alpha^n}_{\Lp^p(B_{2R})}+
\bnorm{\alpha V_\alpha(\Hat{x}_\alpha^n) - \rc^n_{\Hat{v}}
-\widetilde\cI\, \varphi_\alpha^n}_{\Lp^p(B_{2R})}\Bigr)\,.
\end{equation*}
We start with the $\Lp^p$ estimate, say with $p=\frac{d}{d-r}$ for $r\in(1,2)$.
If (a) holds, then
$\norm{\widetilde I \varphi_\alpha^n}_{\Lp^p(B_{R}(x))}\le
\bm\nu\, \norm{\varphi_\alpha^n}_{\Lp^p(B_{R+R_\circ}(x))}$
by the Minkowski integral inequality,
where $R_\circ$ is such that the support of $\nu$ is contained in $B_{R_\circ}$,
while in case (b) we use the technique in the proof of \cref{C2.8}.
Using the compactness of the
embedding $\Sob^{2,p}(B_{R})\hookrightarrow\Lp^q(B_{R})$
for $p\le q <\frac{pd}{d-2p}$, we choose $q=\frac{pd}{d-rp}$
to improve the estimate to a new $p=\frac{d}{d-2r}$.
Continuing in this manner,  in at most $d-1$ steps
we obtain
\begin{equation*}
\sup_{n\ge\varrho_{\Hat{v}}}\,\sup_{\alpha\in(0,1)}\;
\norm{\varphi_\alpha^n}_{\Sob^{2,p}(B_R)}\;<\;\infty
\end{equation*}
for any $p>d$ and $R>0$.
Letting first $n\to\infty$, and then $\alpha\searrow0$, along an appropriate
subsequence, we obtain a solution to \cref{ET3.7A} as claimed.
The rest follow as in the proof of \cref{T3.3}.
\end{proof}

\begin{theorem}\label{T3.8}
Grant the hypotheses of \cref{T3.7}.
Then the conclusions of \cref{T3.3} hold.
Moreover, provided $V$ is inf-compact, a control
$v\in\Usm$ is optimal if and only if it satisfies \cref{ET3.5A}.
\end{theorem}

\begin{proof}
Note that the only place we use the assumption $\Lyap\in\Cc^2(\Rd)$ in the proof
of \cref{T3.3} is to assert that $\widetilde\cI\Lyap\in\Lpl^\infty(\Rd)$.
Thus, under (a), or (b) of \cref{T3.7}, if we select $\Hat{v}\in\Ussm$
such that $\varrho_{\Hat{v}}\le\varrho_*+\epsilon_\circ$, then
the Poisson equation in \cref{ET3.7A} can be used in lieu
\ref{H2}, and the conclusions of \cref{T3.3} follow.
We next show that any $v\in\Usm$ which satisfies \cref{ET3.5A} is stable.
We adapt the technique which is used in
\cite[Theorem~1.2]{Bogachev-01d} for a local operator,
to construct an infinitesimal invariant measure $\mu_v$.
Let $\widetilde\Lg^{*}_v$ be the operator in \cref{PT2.5C}, and set
$\widehat\cI u(x) \df \int_{\Rd}u(x-y) \nu(\D{y})$ if $\nu_x$ is translation invariant; 
otherwise, under hypothesis
(b) of \cref{T3.7}, we define
$\widehat\cI u(x) \df\int_{\Rd} \psi_{x-y}(y)\,u(x-y)\,\D{y}$.
Consider the solution $\upphi_k$ of the Dirichlet problem
$\widetilde\Lg^{*}_v\upphi_k + \widehat\cI \upphi_k = 0$ on $B_k$,
with $\upphi_k$ equal to a positive constant $c_k$
on $B_k^c$.

Concerning the solvability of the Dirichlet problem,
 note that for $f\in\Lp^2(B_k)$, the problem
$\widetilde\Lg^{*}_v u= -\widehat\cI f$ on $B_k$, with $f=u=c_k$
on $B_k^c$,  has a unique solution $u\in\Sob^{2,2}(B_k)$, which obeys the
estimate $\norm{u}_{\Sob^{2,2}(B_k)}\le\kappa(1+\norm{u}_{\Lp^2(B_k)}+
\norm{\widehat\cI f}_{\Lp^2(B_k)})$ for some constant $\kappa$.
Thus we can combine \cref{C2.8}, the compactness of the
embedding $\Sob^{2,2}(B_{R})\hookrightarrow\Lp^q(B_{R})$ for $q=\frac{2d}{d-1}$,
and the Leray--Schauder fixed point theorem to assert
the existence of a solution $\upphi_k\in\Sob^{2,2}(B_k)$ as claimed
in the preceding paragraph.
The solutions $\upphi_k$ are nonnegative by the weak maximum principle
\cite[Theorem~8.1]{GilTru}.
We choose the constant $c_k$
so that $\int_{B_k} \upphi_k(x)\,\D{x}=1$.

We improve the regularity of $\upphi_k$ by 
following the proofs of \cref{T2.5,C2.8}, and show
that for any $n>0$, there exists $N(n)\in\NN$
such that the sequence $\{\upphi_k\,\colon k>N(n)\}$
is H\"older equicontinuous on the ball $B_n$.
Let $R=R(n)>0$ be such that $V(x)> R+1$ on $B_n^c$.
It is always possible to select such $R(n)$ in a manner
that $R(n)\to\infty$ as $n\to\infty$ by the assumption that $V$ is inf-compact.
Employing the function $\chi^{}_R(V)$ as in the proof of \cref{L3.1}
and using \cref{ET3.3A}, it follows that
$\int_{B_{R(n)}} \rc_v(x)\,\upphi_k(x)\,\D{x} \le \varrho_*$
for all $k>N(n)$ and $n\in\NN$.
This implies that
$\int_{\sB_\circ}\upphi_k(x)\,\D{x} \ge
\frac{2\epsilon_\circ}{\varrho_*+2\epsilon_\circ}$ for all large
enough $k$.
By the Arzel\`a--Ascoli theorem combined with Fatou's lemma, $\upphi_k$ converges
along a subsequence to some positive, locally H\"older continuous
$\upphi\in\Lp^1(\Rd)$ uniformly on compact sets,
which is a generalized solution of \cref{PT2.5B}, and thus
satisfies $\int_{\Rd} f(x) \upphi(x)\,\D{x}=0$ for all $f\in\sC$.
Thus, after normalization, $\upphi$ is the density of an infinitesimal
invariant measure.
Therefore, $v\in\Ussm$, and the rest follows by \cref{T3.5}.
\end{proof}

%

\section{A jump diffusion model}\label{S4}

In this section,
we consider a jump diffusion process $X=\{X_t\colon t\ge 0\}$ in $\RR^d$, $d\ge 2$,
defined by the It{\^o} equation
\begin{equation}\label{E-sde}
\D{X_t} \;=\;
b(X_t,Z_t)\,\D{t} + \upsigma(X_t)\,\D{W_t} + \D{L_t}\,, \quad X_0 = x\in \RR^d\,.
\end{equation}
Here, $W=\{W_t,\, t\ge 0\}$ is a $d$-dimensional standard Wiener process, and
$L = \{L_t,\, t\ge 0\}$ is a L\'evy process such that 
$\D L_t = \int_{\RR^m_*} g(X_{t-},\xi)\, \widetilde\cN(\D t, \D \xi)$, 
where $\widetilde\cN$
is a martingale measure in $\RR^m_*=\RR^m\setminus \{0\}$,
$m\ge 1$, corresponding to a standard Poisson random measure $\cN$.
In other words,
$\widetilde\cN(t,A) = \cN(t,A) - t \varPi(A)$ with
$\Exp[\cN(t,A)] = t \varPi(A)$ for any
$A \in\Bor(\RR^m)$, with $\varPi$ a $\sigma$-finite measure on $\RR^m_*$,
and $g$ a measurable function. 

The processes $W$ and $\cN$ are defined on a complete probability space
$(\Omega,\sF,\Prob)$.
Assume that the initial condition $X_0$, $W_0$, and $\cN(0,\cdot)$
are mutually independent. 
The control process $Z = \{Z_t,\, t\ge 0\}$ takes values in a compact,
metrizable space $\Act$,
is  $\mathfrak{F}_t$-adapted, and \emph{non-anticipative}:
for $s < t$, $\bigl(W_{t} - W_{s},\, \cN(t,\cdot)-\cN(s,\cdot)\bigr)$ is independent of
\begin{equation*}
\sF_{s} \;\df\; \text{the completion of~}
\sigma\{X_{0}, Z_{r}, W_{r},\cN(r,\cdot)\,\colon\, r\le s\}
\text{~relative to~} (\sF,\Prob)\,.
\end{equation*} 
Such a process $Z$ is called an \emph{admissible control} and we denote the set
of admissible controls by $\Uadm$.

\subsection{The ergodic control problem for the jump diffusion}

Let $\rc\colon\RR^d\times\Act \mapsto\RR_+$ denote the running cost function,
which is assumed to satisfy \cref{E-nm}.

For an admissible control process $Z\in\Uadm$, we 
consider the \emph{ergodic cost} defined by
\begin{equation*}
\Tilde\varrho^{\vphantom{\frac{1}{2}}}_{Z}(x)
\;\df\; \limsup_{T\rightarrow\infty}\;\frac{1}{T}\;
\Exp_x^Z\biggl[\int_{0}^{T}\rc(X_t,Z_t)\,\D{t}\biggr]\,.
\end{equation*}
Here $\Exp_x^Z$ denotes the expectation operator corresponding to
the process controlled under $Z$, with initial condition
$X_0 = x \in \RR^d$.
The ergodic control problem seeks to minimize the ergodic cost over
all admissible controls.
We define 
$\Tilde\varrho_*(x)\;\df\;\inf_{Z\in\Uadm}\;
\Tilde\varrho^{\vphantom{\frac{1}{2}}}_{Z}(x)$.
As we show in \cref{T4.3}, this infimum is realized with a stationary
Markov control,
and $\Tilde\varrho_*(x)=\varrho_*$, with $\varrho_*$ as defined in
\cref{S2.1}, so it does not depend on $x$.

\subsection{Assumptions on the parameters and the running cost}\label{S4.2}

We impose the following set of assumptions on the data
which guarantee the existence of a solution to the  It{\^o} equation \cref{E-sde}
(see, e.g., \cite{ABG12,GS72}).
These augment and replace \cref{A1.1},
and are assumed throughout this section by default.
In these hypotheses,
 $C_R$ is a positive constant, depending on $R\in(0,\infty)$.
Also $a\df\frac{1}{2}\upsigma\upsigma'$,
$\RR^m_* \df \RR^m\setminus\{0\}$, and
$\lVert M\rVert\df\bigl(\trace\, MM'\bigr)^{\nicefrac{1}{2}}$
denotes the Hilbert--Schmidt norm of a
$d\times k$ matrix $M$ for $d, k \in \NN$.
\begin{gather*}
\abs{b(x,z)-b(y,z)}^2 + \norm{\upsigma(x) - \upsigma(y)}^2 + \int_{\RR^m_*} 
\abs{g(x,\xi) - g(y,\xi)}^2 \varPi(\D \xi)\\
\mspace{100mu}
+\,\abs{\rc(x,z)-\rc(y,z)}^2\;\le\; C_R\abs{x-y}^2
\qquad\forall\,x,y\in B_R\,,\ \forall\,z\in\Act\,,\\
\bigl\langle x,b(x,z)\bigr\rangle^+
+ \norm{\upsigma(x)}^2 + \int_{\RR^m_*} \abs{g(x,\xi)}^2 \varPi(\D \xi)
\;\le\; C_1(1+\abs{x}^2)\qquad\forall\, (x,z)\in\Rd\times\Act\,,\\
\sum_{i,j}^{}a^{ij}(x)\zeta_i\zeta_j\;\ge\; (C_{R})^{-1} \abs{\zeta}^2
\quad\forall \zeta\in\Rd\,,\ \forall\,x\in B_R\,.
\end{gather*}

The measure $\nu_x$ in \cref{E-Ag} then takes the form
$\nu_{x}(A) = \varPi\bigl(\{\xi\in\RR_{*}^{m}\,\colon g(x,\xi)\in A\}\bigr)$,
and it clearly  satisfies
$\int_{\Rd} \abs{y}^2\,\nu_x(\D{y})\;<\;C_R\abs{x}^2$.
Note that for this model $\bm\nu=\nu_x(\Rd)$ is constant.
It is evident that if $g(x,\xi)$ does not depend on $x$, 
then $\nu_x$ is translation invariant.

\subsection{Existence of solutions}
For any admissible control $Z_t$,
the It\^o equation in \cref{E-sde} has a unique strong solution \cite{GS72},
is right-continuous w.p.$1$, and is a strong Feller process.
On the other hand, if $Z_t$ is a Markov control, i.e., if it takes the
form $Z_t=v(t,X_t)$ for some Borel measurable function $v\colon\RR_+\times\Rd$,
then it follows from the results in \cite{Gyongy-96} that,
under the assumptions in \cref{S4.2}, the diffusion
\begin{equation}\label{E-sde-aux}
\D{\Tilde X_t} \;=\;
b(\Tilde X_t,v(t,\Tilde X_t))\,\D{t} + \upsigma(\Tilde X_t)\,\D{W_t}\,,
\quad X_0 = x\in \RR^d
\end{equation}
has a unique strong solution.
As shown in \cite{Skorokhod-89}, since the
the L\'evy measure is finite,
the solution of \cref{E-sde} can be constructed in a piecewise fashion
using the solution of \cref{E-sde-aux} (see also \cite{Li-05}).
It thus follows that, under a Markov control, \cref{E-sde-aux} has a unique
strong solution.
In addition, its transition probability has positive mass.

Of fundamental importance in the study of functionals of $X$ is
It\^o's formula.
For $f\in\Cc^{2}(\RR^{d})$ and $Z_s$ an admissible control,
it holds that
\begin{equation}\label{E-Ito}
f(X_{t}) \;=\; f(X_{0}) + \int_{0}^{t}\Ag f(X_{s},Z_s)\,\D{s}
+ \mathscr{M}_{t}\quad\text{a.s.},
\end{equation}
with $\Ag$ as in \cref{E-Ag}, and
\begin{align}\label{E-mart}
\mathscr{M}_{t} &\;\df\; \int_{0}^{t}\bigl\langle\nabla f(X_{s}),
\upsigma(X_{s})\,\D{W}_{s}\bigr\rangle  \\ 
&\mspace{100mu}
 + \int_{0}^{t}\int_{R^m_*}
\Bigl(f\bigl(X_{s-} + g(X_{s-},\xi)\bigr)- f(X_{s-})\Bigr) \,
\widetilde\cN(\D s, \D \xi)\nonumber
\end{align}
is a local martingale.
Krylov's extension of It\^o's formula \cite[p.~122]{Krylov}
shows that \eqref{E-Ito} is valid for functions $f$ in the local
Sobolev space $\Sobl^{2,p}(\RR^{d})$, $p\ge d$.

Recall that, in the context of diffusions,
a control $v\in\Usm$ is called \emph{stable}
if the process $X$ under $v$ is positive Harris recurrent.
This is of course equivalent to the existence of
an invariant probability measure for $X$,
and it follows by the Theorem in \cite{PE82} that
$\mu_v$ is an invariant probability measure for the diffusion
if and only if it is infinitesimally invariant for the operator $\Ag$
in the sense of \eqref{E-infmeas}.
Thus the two notions of stable controls agree.

\subsection{Existence of an optimal stationary Markov control}

\begin{definition}
For $Z \in \Uadm$ and $x \in \RR^d$, we define the mean empirical measures
$\{\Bar{\zeta}^Z_{x,t}\,\colon t>0\}$, and (random) empirical measures
$\{\zeta^Z_{t}\,\colon t>0\}$, by
\begin{equation}\label{E-emp}
\Bar{\zeta}^Z_{x,t}(f)\;=\;
\int_{\RR^d\times\Act}f(x,z)\,\Bar{\zeta}^Z_{x,t}(\D{x},\D{z})
\;\df\; \frac{1}{t}\int_{0}^{t}
\Exp_x^Z\biggl[\int_{\Act} f(X_s,z)\,Z_s(\D{z})\biggr]\,\D{s}\,,
\end{equation}
and $\zeta^Z_{x,t}$ as in \cref{E-emp} but without the expectation
$\Exp_x^Z$,
respectively, for all $f\in\Cc_b(\RR^d\times\Act)$.
\end{definition}

We let $\overline{\RR}^{d}$ denote the one-point compactification of $\RR^{d}$, and
we view $\RR^{d}\subset\overline\RR^{d}$ via the natural imbedding.
As a result, $\Pm(\RR^{d}\times\Act)$ is viewed as a subset
of $\Pm(\overline\RR^{d}\times\Act)$.
Let $\Bar\eom$ denote the closure of $\eom$ in $\Pm(\overline\RR^{d}\times\Act)$.

\begin{lemma}\label{L4.2}
Almost surely, every limit $\Hat{\zeta} \in  \Pm(\overline{\RR}^d\times\Act)$
of $\zeta^Z_{t}$ as $t\to\infty$ takes the form 
$\Hat{\zeta}= \delta \zeta' + (1-\delta) \zeta''$ for some $\delta \in [0,1]$,
with $\zeta' \in \eom$ and $\zeta'' (\{\infty\} \times \Act) =1$.
The same claim  holds for the mean empirical measures,
without the qualifier `almost surely'.
\end{lemma}

\begin{proof}
Write $\Hat{\zeta}=\delta \zeta' + (1-\delta) \zeta''$
 for some
$\zeta'\in \Pm(\RR^d\times \Act)$, and
$\zeta'' (\{\infty\} \times \Act) =1$.
For $f\in\sC$, applying It\^o's formula, we obtain
\begin{equation*}
\frac{f(X_t) - f(X_0)}{t} \; =\; \frac{1}{t}\int_{0}^{t} \Ag_{Z_s}f(X_s)\,\D{s}
+ \frac{1}{t}\,\mathscr{M}_{t}\,,
\end{equation*}
where $\mathscr{M}_t$ is given in \cref{E-mart}.
As shown in the proof of \cite[Lemma 3.4.6]{ABG12}, we have
$\frac{1}{t}\int_{0}^{t} \bigl\langle\nabla f(X_{s}),
\upsigma(X_{s})\,\D{W}_{s}\bigr\rangle \to0$
a.s.\  as $t\rightarrow\infty$.

Define 
\begin{equation}\label{EL4.2A}
M_{1,t} \;\df\; \int_{0}^{t}\int_{R^m_*}
\Bigl(f\bigl(X_{s-} + g(X_{s-},\xi)\bigr)- f(X_{s-})\Bigr)\,
\cN(\D s, \D \xi)\,,
\end{equation}
and $M_{2,t}$ analogously by replacing $\cN(\D s, \D \xi)$ by 
$\varPi(\D{\xi})\,\D{s}$ in \cref{EL4.2A}.
Note that the second integral
in \cref{E-mart}, denoted as $M_t$, is a square integrable martingale,
and takes the form
$M_t= M_{1,t} - M_{2,t}$.  
Since $f$ is bounded on $\RR^d$ and $\varPi$ is a finite measure,
we have $\langle M_1\rangle_t \le C_1\,\cN(t,\RR^m_*)$,
and $\langle M_2\rangle_t \le C_2t$ 
for some positive constants $C_1$ and $C_2$.
Since  $\langle M\rangle_t \le \langle M_1\rangle_t + \langle M_2\rangle_t$,
then by Proposition 7.1 in \cite{Ross14} we obtain
$\limsup_{t\rightarrow\infty}\frac{\langle M\rangle_t}{t}
<\infty$ a.s.
For the discrete parameter square-integrable martingale $\{M_n\,\colon n\in\NN\}$,
it is well-known that 
$\lim_{n\rightarrow\infty} \frac{M_n}{\langle M\rangle_n} = 0$ a.s.
on the event $\{\langle M\rangle_\infty = \infty\}$.
Thus, we obtain
\begin{equation}\label{EL4.2B}
\lim_{n\rightarrow\infty} \frac{M_n}{n} \;=\; 0 \quad \text{a.s.}
\end{equation}
on the event $\{\langle M\rangle_\infty = \infty\}$.
Since $f$ is bounded, then for some constant $C>0$, we have
\begin{equation}\label{EL4.2C}
\sup_{t\in[n,n+1]}\frac{\abs{M_t -M_n}}{n}\;\le\;
\frac{C}{n}\,\bigl(\cN(n+1,\RR^m_*)-\cN(n,\RR^m_*)+ 1\bigr)
\;\xrightarrow[n\to\infty]{}\;0\,,
\end{equation}
and \cref{EL4.2B}--\cref{EL4.2C} imply that
$\lim_{t\rightarrow\infty}\frac{1}{t}\,M_t \;\to\; 0$ a.s.\
on the event $\{\langle M\rangle_\infty = \infty\}$.

Next, we examine convergence on the event $\{\langle M\rangle_\infty < \infty\}$.
It is well-known that a
square-integrable martingale $\{ M_n\,\colon n\in\NN\}$
with quadratic variation $\langle M\rangle$ satisfies 
$\{\langle M\rangle_\infty<\infty\}\subset \{M_n\to\ \}$ a.s.,
where we write $\{M_n\to\ \}$ for the event on which $(M_n)$ converges to
a real-valued limit \cite[Theorem~2.15]{Hall-Heyde}.
Thus \cref{EL4.2B} holds on the event $\{\langle M\rangle_\infty < \infty\}$, 
and it then follows by \cref{EL4.2C} that
$\lim_{t\rightarrow\infty}\frac{1}{t}\,M_t \;\to\; 0$ a.s.

Thus we have shown that $\lim_{t\rightarrow\infty}\frac{1}{t}\,\mathscr{M}_t \to 0$ a.s.,
and the claims of the lemma then follow as in the proof of \cite[Theorem~3.4.7]{ABG12}.
\end{proof}

\begin{theorem}\label{T4.3}
There exists an optimal
control $v \in\Ussm$ for the ergodic problem. In addition,
every stationary Markov optimal control $v_*$ is in $\Ussm$, and is pathwise optimal
in somewhat stronger sense,
i.e., it satisfies
\begin{equation}\label{ET4.3A}
\liminf_{T\rightarrow\infty}\;\frac{1}{T}\;
\biggl[\int_{0}^{T}\rc(X_t,Z_t)\,\D{t}\biggr] \;\ge\;
\limsup_{T\rightarrow\infty}\;\frac{1}{T}\;
\biggl[\int_{0}^{T}\rc\bigl(X_t,v_*(X_t)\bigr)\,\D{t}\biggr]\;=\;\varrho_*
\end{equation}
a.s.\ for any admissible control $Z_t$.
\end{theorem}

\begin{proof}
Define $\Hat\varrho_*\df\inf_{\uppi\in\eom}\,\uppi(\rc)$.
Following the proof of \cite[Theorem~3.4.5]{ABG12},
we have 
$\Hat\varrho_*=\uppi_{v_*}(\rc)$ for some $v_*\in\Ussm$.
Also, \cref{ET4.3A} holds by 
\cref{L4.2} and the proof in \cite[Theorem~3.4.7]{ABG12}.
\end{proof}


\subsection{The ergodic HJB equation}

We summarize the results in the following theorem.

\begin{theorem}
We assume \textup{\ref{H2}} for some $\Hat{v}\in\Usm$.
Then we have the following:
\begin{itemize}
\item[\textup{(}a\/\textup{)}]
There exists a unique function $V\in\Sobl^{2,p}(\Rd)$, $p>d$, with $V(0)=0$,
which is bounded from below in $\Rd$ and solves
$\min_{z\in\Act}\;\bigl[\Ag_{z}\,V(x) + \rc(x,z)\bigr] = \varrho$,
with $\varrho=\varrho_*$.
For $\varrho<\varrho_*$, there is no such solution.
Moreover, if $\nu_x$ has locally compact support \textup{(}see \cref{D2.7}\textup{)},
then $V\in\Cc^2(\Rd)$.

\item[\textup{(}b\/\textup{)}]
A control $v\in\Usm$ is optimal if and only if it satisfies
\begin{equation}\label{ET4.4B}
b^i_{v}(x)\,\partial_i V(x)+\rc_{v}(x)
\;=\;  \inf_{z\in\Act}\;\bigl[b^i(x,z)
\partial_i V(x)+\rc(x,z)\bigr]\quad\text{a.e.\ } x\in\Rd\,.
\end{equation}

\item[\textup{(}c\/\textup{)}]
The solution $V$ has the stochastic representation
\begin{equation*}
V(x) \;=\; \lim_{r\searrow0}\;\inf_{v\in\Ussm}\; \Exp^v_x\biggl[\int_0^{\uuptau_r}
(\rc_v(X_t)-\varrho_*\bigr)\,\D{t}\biggr]\,.
\end{equation*}
\end{itemize}
\end{theorem}

\begin{proof}
Under the assumptions in \cref{S4.2}, it is straighforward to establish
\cref{T3.2}.
Thus, part (a) follows from \cref{T3.3,R3.4}.
Using the It\^o formula, one can readily show that any
$v$ which satisfies \cref{ET4.4B} is stable and optimal.
The necessity part of (b) follows by \cref{T3.5}.
Part (c) can be established by following the proof of
Lemma~3.6.9 in \cite{ABG12}.
\end{proof}

\section{Concluding remarks}

The results in this paper extend naturally to models under uniform stability,
in which case, of course, we do not need to assume that $\rc$ is coercive.
Suppose that there exist nonnegative functions $\Psi\in\Cc^2(\Rd)$,
and $h\colon\Rd\times\Act$, with $h\ge1$ and locally bounded, satisfying
\begin{equation}\label{E-Lyap}
\Ag_z \Psi (x) \;\le\; \kappa\,\Ind_{\sB}(c) - h(x,z)
\qquad\forall\,(x,z)\in\Rd\times\Act\,,
\end{equation}
for some constant $\kappa$ and a ball $\sB\subset\Rd$.
In addition, suppose that either $\rc$ is bounded, or that $\abs{\rc}$ grows slower
than $h$.
Under \cref{E-Lyap}, the jump diffusion is positive recurrent under any
stationary Markov control, and the collection of ergodic occupation measures
is tight.
Using $\Psi$ as a barrier, all the results in \cref{S4}
can be readily obtained, and moreover, for any $v\in\Usm$, the Poisson
equation $\Ag_v \Phi = \rc_v - \varrho_v$ has a solution in
$\Sobl^{2,p}(\Rd)$, for any $p>1$, which is unique, up to an additive constant,
in the class
of functions $\Phi$ which satisfy $\abs{\Phi}\le C(1+h_v)$ for some constant $C$.

We have not considered allowing the jumps to be control dependent, primarily because
this is not manifested in the queueing network model motivating this work,
but also because
this would require us to introduce various assumptions on the regularity
of the jumps and the L\'evy measure (see, e.g., \cite{Menaldi-99}).
This, however, is an interesting problem for future work.

In conclusion, what we aimed for in this work, was to study the ergodic
control problem for jump diffusions controlled through the drift via
analytical methods, and under
minimal assumptions on the (finite) L\'evy measure and the parameters.

\section*{Acknowledgments}
This work was supported in part by the
National Science Foundation through grants DMS-1540162, DMS-1715210,
CMMI-1538149, and DMS-1715875,
in part by the Army Research Office
through grant number W911NF-17-1-0019, and in part
Office of Naval Research
through grant number N00014-16-1-2956.

\def\polhk#1{\setbox0=\hbox{#1}{\ooalign{\hidewidth
  \lower1.5ex\hbox{`}\hidewidth\crcr\unhbox0}}}

\end{document}